\documentclass[final,onefignum,onetabnum]{siamart220329}

\usepackage{braket,amsfonts,amssymb}

\usepackage{array}

\usepackage[caption=false]{subfig}
\usepackage{pgfplots}

\newsiamthm{claim}{Claim}
\newsiamremark{remark}{Remark}
\newsiamremark{hypothesis}{Hypothesis}
\crefname{hypothesis}{Hypothesis}{Hypotheses}

\newsiamthm{assumption}{Assumption}
\crefname{assumption}{Assumption}{Assumptions}

\usepackage{algorithmic}

\usepackage{graphicx,epstopdf}

\Crefname{ALC@unique}{Line}{Lines}

\usepackage{amsopn}

\usepackage{xspace}
\usepackage{bold-extra}
\usepackage[most]{tcolorbox}

\usepackage{xspace}
\usepackage{bold-extra}
\usepackage[most]{tcolorbox}

\colorlet{texcscolor}{blue!50!black}
\colorlet{texemcolor}{red!70!black}
\colorlet{texpreamble}{red!70!black}
\colorlet{codebackground}{black!25!white!25}

\lstdefinestyle{siamlatex}{%
  style=tcblatex,
  texcsstyle=*\color{texcscolor},
  texcsstyle=[2]\color{texemcolor},
  keywordstyle=[2]\color{texemcolor},
  moretexcs={cref,Cref,maketitle,mathcal,text,headers,email,url},
}

\tcbset{%
  colframe=black!75!white!75,
  coltitle=white,
  colback=codebackground, %
  colbacklower=white, %
  fonttitle=\bfseries,
  arc=0pt,outer arc=0pt,
  top=1pt,bottom=1pt,left=1mm,right=1mm,middle=1mm,boxsep=1mm,
  leftrule=0.3mm,rightrule=0.3mm,toprule=0.3mm,bottomrule=0.3mm,
  listing options={style=siamlatex}
}

\newtcblisting[use counter=example]{example}[2][]{%
  title={Example~\thetcbcounter: #2},#1}

\newtcbinputlisting[use counter=example]{\examplefile}[3][]{%
  title={Example~\thetcbcounter: #2},listing file={#3},#1}

\DeclareTotalTCBox{\code}{ v O{} }
{ %
  fontupper=\ttfamily\color{black},
  nobeforeafter,
  tcbox raise base,
  colback=codebackground,colframe=white,
  top=0pt,bottom=0pt,left=0mm,right=0mm,
  leftrule=0pt,rightrule=0pt,toprule=0mm,bottomrule=0mm,
  boxsep=0.5mm,
  #2}{#1}

\patchcmd\newpage{\vfil}{}{}{}
\usepackage[sort,compress]{cite}
\usepackage{yhmath,mathtools,bm,comment,amssymb,enumitem,xcolor,array,float,tikz,pgfplots,siunitx,booktabs,accents,multirow,epsfig,mathrsfs,comment}

\newcommand{\phim}{\hat\phi}
\newcommand{\hzeta}{\hat\zeta}
\newcommand{\nablax}{\nabla_{\!x}}
\newcommand{\nablaq}{\nabla_{\!q}}

\newcommand{\init}{\hpsi_k^0}

\newcommand{\pt}{\partial_t}
\newcommand{\eps}{\varepsilon}
\newcommand{\pta}{\p_t^\alpha}
\newcommand{\ptb}{\p_t^{1-\alpha}}
\newcommand{\ga}{g_{\alpha}}
\newcommand{\gb}{g_{1-\alpha}}

\newcommand{\ds}{\,\textup{ds}}

\renewcommand{\div}{\textup{div}}

\newcommand{\dd}{\mathop{}\!\mathrm{d}}
\renewcommand{\d}{\mathop{}\!\mathrm{d}}
\newcommand{\ddt}{\frac{\dd}{\dd\mathrm{t}}}
\newcommand{\dt}{\,\textup{d}t}
\newcommand{\p}{\partial}

\newcommand{\R}{\mathbb{R}} %

\renewcommand{\rho}{\varrho}

\newcommand{\con}{\hookrightarrow}
\newcommand{\com}{\mathrel{\mathrlap{{\mspace{4mu}\lhook}}{\hookrightarrow}}}

\newcommand{\W}{\mathcal{W}} %
\renewcommand{\H}{\mathcal{H}} %
\newcommand{\C}{\mathcal{C}} %
\newcommand{\HS}{\mathcal{H}} %
\newcommand{\HSV}{\mathcal{V}} %

\newcommand{\longweak}{\relbar\joinrel\rightharpoonup}

\newcommand{\dq}{\,\text{d}q}
\newcommand{\hpsi}{{\widehat{\psi}}}
\newcommand{\hphi}{{\widehat{\phi}}}
\newcommand{\hY}{{\widehat{\mathcal{Y}}}}
\newcommand{\hX}{{\widehat{\mathcal{X}}}}
\newcommand{\hZ}{{\widehat{\mathcal{Z}}}}
 
\allowdisplaybreaks
\renewcommand{\hat}{\widehat}
\usepackage{verbatim}

\begin{tcbverbatimwrite}{tmp_\jobname_header.tex}
\title{Analysis of a dilute polymer model with a time-fractional derivative\thanks{Submitted to the editors \today.
\funding{MF is supported by the State of
Upper Austria.}}}

\author{Marvin Fritz\thanks{Computational Methods for PDEs, Johann Radon Institute for Computational and Applied Mathematics, Altenberger Str. 69, 4040 Linz, Austria (\email{marvin.fritz@ricam.oeaw.ac.at}).}
\and Endre S\"uli\thanks{Mathematical Institute, University of Oxford, Andrew Wiles Building, Woodstock Road, Oxford OX2 6GG, United Kingdom (\email{endre.suli@maths.ox.ac.uk}).}
\and Barbara Wohlmuth\thanks{Department of Mathematics, Technical University of Munich, Boltzmannstr. 3,
85748 Garching bei M\"unchen, Germany (\email{wohlmuth@ma.tum.de})}}

\headers{Analysis of a dilute polymer model}{Marvin Fritz, Endre S\"uli, Barbara Wohlmuth}
\end{tcbverbatimwrite}
\title{Analysis of a dilute polymer model with a time-fractional derivative\thanks{Submitted to the editors \today.
\funding{MF is supported by the State of
Upper Austria.}}}

\author{Marvin Fritz\thanks{Computational Methods for PDEs, Johann Radon Institute for Computational and Applied Mathematics, Altenberger Str. 69, 4040 Linz, Austria (\email{marvin.fritz@ricam.oeaw.ac.at}).}
\and Endre S\"uli\thanks{Mathematical Institute, University of Oxford, Andrew Wiles Building, Woodstock Road, Oxford OX2 6GG, United Kingdom (\email{endre.suli@maths.ox.ac.uk}).}
\and Barbara Wohlmuth\thanks{Department of Mathematics, Technical University of Munich, Boltzmannstr. 3,
85748 Garching bei M\"unchen, Germany (\email{wohlmuth@ma.tum.de})}}

\headers{Analysis of a dilute polymer model}{Marvin Fritz, Endre S\"uli, Barbara Wohlmuth}

\ifpdf
\hypersetup{ pdftitle={Analysis of a dilute polymer model with time-fractional derivatives} }
\fi

\begin{document}
\maketitle

\begin{tcbverbatimwrite}{tmp_\jobname_abstract.tex}
\begin{abstract}
We investigate the well-posedness of a coupled Navier--Stokes--Fokker--\-Planck system with a time-fractional derivative. Such systems arise in the kinetic theory of dilute solutions of polymeric liquids, where the motion of noninteracting polymer chains in a Newtonian solvent is modelled by a stochastic process exhibiting 
power-law waiting time, in order to capture subdiffusive processes associated with non-Fickian diffusion.
We outline the derivation of the model from a subordinated Langevin equation. The elastic properties of the polymer molecules immersed in the solvent are modelled by a finitely extensible nonlinear elastic (FENE) dumbbell model, and the drag term in the Fokker--Planck equation is assumed to be corotational.
We prove the global-in-time existence of large-data weak solutions to this time-fractional model of order $\alpha \in (\tfrac12,1)$, and derive an energy inequality satisfied by weak solutions.
\end{abstract}

\begin{keywords}
Existence of weak solutions; Navier--Stokes--Fokker--Planck system; dilute polymer model; Hookean and FENE-type bead-spring-chain model; Riemann--Liouville fractional derivative; time-fractional PDE
\end{keywords}

\begin{MSCcodes}
35Q30, 35Q84, 35R11, 60G22, 82C31, 82D60
\end{MSCcodes}
\end{tcbverbatimwrite}
\begin{abstract}
We investigate the well-posedness of a coupled Navier--Stokes--Fokker--\-Planck system with a time-fractional derivative. Such systems arise in the kinetic theory of dilute solutions of polymeric liquids, where the motion of noninteracting polymer chains in a Newtonian solvent is modelled by a stochastic process exhibiting
power-law waiting time, in order to capture subdiffusive processes associated with non-Fickian diffusion.
We outline the derivation of the model from a subordinated Langevin equation. The elastic properties of the polymer molecules immersed in the solvent are modelled by a finitely extensible nonlinear elastic (FENE) dumbbell model, and the drag term in the Fokker--Planck equation is assumed to be corotational.
We prove the global-in-time existence of large-data weak solutions to this time-fractional model of order $\alpha \in (\tfrac12,1)$, and derive an energy inequality satisfied by weak solutions.
\end{abstract}

\begin{keywords}
Existence of weak solutions; Navier--Stokes--Fokker--Planck system; dilute polymer model; Hookean and FENE-type bead-spring-chain model; Riemann--Liouville fractional derivative; time-fractional PDE
\end{keywords}

\begin{MSCcodes}
35Q30, 35Q84, 35R11, 60G22, 82C31, 82D60
\end{MSCcodes}

		\section{Introduction}
	
This paper is concerned with the existence of weak solutions to a system of nonlinear partial differential equations that arises in the kinetic theory of dilute solutions of polymeric fluids. Within this class of models we focus on finitely-extensible nonlinear elastic, FENE-type, dumbbell models with a corotational drag term. In contrast to previous literature on the analysis of these models we assume power law waiting times in the derivation of the system, which results in the appearance of a time-fractional derivative in the Fokker--Planck equation describing the evolution of the probability density function. This raises new questions about the study of well-posedness, and we provide rigorous results concerning the existence of global-in-time weak solutions to the system of partial differential equations featuring in the model.

Dilute polymer models are derived and extensively described in the monograph \cite{bird1987dynamics2} and in the book by \"Ottinger \cite{ottinger2012stochastic}; see also \cite{suli2018mckeanvlasov} for a mathematically rigorous derivation of the Hookean bead-spring-chain model from Brownian dynamics. We also refer to the papers \cite{lemou2002viscoelastic,herrchen1997a} for a comparison of several FENE-type dumbbell models. Such systems are of microscopic-macroscopic type since they involve a coupling of the (macroscopic) Navier--Stokes equations for the description of incompressible fluid flow and the Fokker--Planck equation for the microscopic processes associated with 
the statistical properties of polymer molecules immersed in the fluid. 
Concerning the weak and strong well-posedness of FENE-type models, we refer to the works \cite{jourdain2004existence,kreml2010on,masmoudi2013global,zhang2006local,renardy1991an}. More general dilute polymer models are analyzed in \cite{barrett2005existence,barrett2007existence,barrett2008existence,barrett2010existence,barrett2010existence2}. Further, we mention the papers \cite{lions2000global,lions2007global,schonbek2009existence,masmoudi2008well,barrett2005existence,barrett2009numerical,debiec2023corotational,lin2008global,busuioc2014fene}, which, similarly to the discussion herein, are concerned with dumbbell models that assume a corotational drag term in the Fokker--Planck equation.  In such models it is supposed that polymer molecules are not stretched by the surrounding solvent, although they are allowed to rotate without stretching; see, for example, \cite{la2020diffusive}.

Time-fractional differential equations have been the focus of considerable attention in the mathematical and engineering literature in recent years. Such equations are nonlocal in time and have an innate history effect. They are of relevance in applications where memory effects are present and hereditary properties of materials are studied; see, for example, the textbooks on viscoelasticity \cite{mainardi2022fractional,yang2020general}, hydrology \cite{su2020fractional}, financial economics \cite{fallahgoul2016fractional}, and mechanical processes \cite{atanackovic2014fractional,pilipovic2014fractional}. The time-fractional Fokker--Planck system, in particular, allows subdiffusive behaviour and has been previously studied in \cite{metzler1999anomalous,metzler1999deriving,metzler2000random,barkai2000continuous,barkai2001fractional,henry2006anomalous,henry2010fractional,henry2010introduction,langlands2008anomalous} with regards to its derivation and applicability. The articles \cite{pinto2017numerical,le2016numerical,le2018a,le2019existence,le2021alpha} have investigated the numerical analysis and the simulation of solutions to the time-fractional Fokker--Planck equation. The time-fractional model considered herein has been explored computationally in \cite{beddrich2023numerical}, albeit in the simpler setting of a linear (Hookean) elastic spring force instead of the FENE spring model that we study here.
 
We employ a spatial Galerkin approximation in conjunction with a compactness argument to prove the existence of weak solutions to the time-fractional Navier--Stokes--Fokker--Planck system under consideration. More specifically, we discretize the system in space and derive appropriate energy bounds, which then enable us to pass to the limit in the discretized system.  Spatial discretizations of dilute polymer models were previously considered  in  \cite{barrett2009numerical,barrett2011finite,barrett2012finite}. In addition, weak solutions to time-fractional PDEs have been investigated using the Galerkin approach in the publications \cite{fritz2021sub,fritz2020time,fritz2021equivalence}. There have also been initial steps in the analysis of a decoupled time-fractional Fokker--Planck equation with time-dependent forces; see the papers \cite{fritz2023well,mclean2020regularity,le2019existence,le2021alpha,mclean2021uniform}. However, the coupling of the time-fractional Fokker--Planck equation to the Navier--Stokes system gives rise to new technical complications, which have not been addressed previously.

In \Cref{Sec:Derivation} we derive the model from the Langevin equation assuming power-law waiting time. In this way time-fractional derivatives in the sense of Riemann--Liouville appear in the associated  Fokker--Planck equation. By mimicking the technique for the derivation of the standard dumbbell model, a time-fractional Navier--Stokes--Fokker--Planck system is obtained.  In \Cref{Sec:Prelim} we introduce several function spaces of Sobolev-type and recall some important results from the theory of fractional derivatives, including chain inequalities and embedding theorems. In \Cref{Sec:Form} we transform the model in order to make it amenable to the subsequent analysis.  We then equip the model with suitable initial and boundary conditions and we make use of the associated Maxwellian to rescale the Navier--Stokes--Fokker--Planck system. In \Cref{Sec:Analysis} we finally state and prove a theorem asserting the existence of large-data global-in-time weak solutions to the model with a time-fractional derivative of order $\alpha \in (\tfrac12,1)$.

 	 	\section{Derivation of the time-fractional FENE-type system} \label{Sec:Derivation}

In this section we derive the time-fractional Navier--Stokes--Fokker--Planck model, admitting both Hookean and FENE-type bead-spring-chains. The classical Hookean bead-spring-chain model is derived from a system of stochastic differential equations; see, for example, the articles \cite{suli2018mckeanvlasov,barrett2007existence} and the thesis \cite[Section 1.3]{ye2018numerical}. The time-fractional Fokker--Planck equation is derived from a Langevin equation in \cite{magdziarz2009stochastic}. We shall emphasize the differences in the derivation of the time-fractional Navier--Stokes--Fokker--Planck model and examine the steps where the time-fractional derivative is introduced. 

In this work, the time-fractional derivative of order $\alpha \in (0,1)$ is understood in the sense of Riemann--Liouville and is given by 
\begin{equation} \label{Eq:RL} 
	\pta w(t) := \pt \int_0^t \frac{(t-s)^{-\alpha}}{\Gamma(1-\alpha)} w(s) \ds,\end{equation} 
where $\Gamma(\alpha):=\int_0^\infty t^{\alpha-1}e^{-t}\dt$ is Euler's Gamma function.

%Later in this section, we shall state suitable assumptions concerning the quantities that appear in the upcoming PDE model and will supplement the system with appropriate initial and boundary conditions.

\subsection{Derivation} We shall idealize each polymer molecule as a pair of massless beads connected by a massless elastic spring. It is assumed that the resulting, so called, \textit{dumbbell} is suspended in a Newtonian solvent, whose motion is governed by the incompressible Navier--Stokes equations for the velocity $u$ and the pressure $p$ of the fluid. Let us denote the position vectors (with respect to an arbitrary, but fixed, reference point in $\mathbb{R}^3$) of the centers of mass of the two beads at time $t$ by $x_i(t) \in \R^3$ for $i \in \{1,2\}$. At time $t$, the center of mass of the dumbbell is at $x_c(t):=\frac12\big(x_1(t)+x_2(t)\big)$ and the elongation (or conformation) vector pointing from $x_1(t)$ to $x_2(t)$ is $q_1(t):=x_2(t)-x_1(t)$. We denote the vector pointing in the opposite direction by $q_2(t):=-q_1(t)$ and assume that $q_1(t)$ and $q_2(t)$ are contained, for all $t \geq 0$, within a given convex open set $D \subset \R^3$ that satisfies $0 \in D$, and $-q \in D$ whenever $q\in D$.

Three kinds of force act on the $i$-th bead in the dumbbell suspended in the fluid: a drag force (Stokes drag) arising from the motion of the bead through the solvent, an elastic spring force, and a random force, which is assumed to be Brownian, modelling the random collisions of the bead with the molecules of the surrounding Newtonian solvent. As each of the two beads is assumed to be massless, Newton's second law implies that
\begin{equation} \label{Eq:Newton} \text{Drag Force}_i +\text{Spring Force}_i + \text{Brownian Force}_i =0,\quad i=1,2.\end{equation} 
To define the drag force 
 acting on the $i$-th bead of the dumbbell suspended in the fluid, we apply  Stokes' law and get
\begin{eqnarray*}
- \zeta\Big(\ddt x_i(t)-u(x_i,t)\Big),\quad i \in \{1,2\}.
\end{eqnarray*} Here $\zeta$ is the friction coefficient and $u(\cdot,\cdot) $ stands for the fluid velocity. We note that $\zeta$ carries the SI unit $[\text{kg}/\text{s}] $ and is linear in the dynamic viscosity $\eta$. For a sphere of radius $R$ it reads $\zeta = 6 \pi R \eta$.

The elastic spring force $F:D \to \R^3$ of the spring connecting the two beads is $$F(q):=HU'(\tfrac12 |q|^2)q,$$ where $U$ is a nonnegative continuously differentiable potential and $H$ is the spring constant having the SI unit $[\text{kg}/\text{s}^2]$.
In the case of the Hookean dumbbell model the spring force is linear and is given by $F(q)=Hq$ with $q\in D=\R^3$ and the corresponding potential is  $U(s)=s$ for $s \in [0,\infty)$, while in the case of a classical FENE model one has, instead, $$D=B_{|q_\text{max}|}(0), \quad F(q)=\frac{Hq}{1-|q|^2/|q_\text{max}|^2}, \quad U(s)=-\frac{|q_\text{max}|^2}{2}\ln\!\bigg(1-\frac{2s}{|q_\text{max}|^2}\bigg)$$ for $q\in D$ and $s\in [0,\tfrac{1}{2}|q_\text{max}|^2)$, where $B_{|q_\text{max}|}(0)$ is an open ball in $\R^3$ with radius $|q_\text{max}|$ centered at the origin and $|q_\text{max}|>0$ is a strict upper bound on the maximal extension to which a dumbbell can be stretched.

The Brownian force acting on the $i$-th bead at time $t$ is denoted by $B_i(t)$  and is formally defined by 
$$B_i(t):= \sqrt{2k_B\mu_T\zeta}\, \frac{\dd W_i(t)}{\text{d}t },\quad i \in \{1,2\},$$
where $k_B$ is the Boltzmann constant in $[\text{kg m}^2 /(\text{s}^2 \text{K})]$,  $\mu_T$ denotes the  temperature in $[\text{K}]$, and $W(t)=\big(W_1(t),W_2(t)\big)^{\mathrm{T}}$ is a vector of two independent Wiener processes. 
As each Wiener process is distributed according to $\mathcal{N} (0,t)$, we find that the SI unit for $ \frac{\dd W_i(t)}{\text{d}t }$ is $[1/\sqrt{\text{s }}]$.

By introducing the notation
\begin{equation} \label{Def:Vectorb} X(t):=\begin{pmatrix} x_1(t) \\ x_2(t) \end{pmatrix}, \quad b(X(t),t):=\begin{pmatrix} u(x_1(t),t)+\zeta^{-1} F(q_1(t)) \\ u(x_2(t),t)+\zeta^{-1} F(q_2(t)) \end{pmatrix},\end{equation}
and dividing \cref{Eq:Newton} by $\zeta$, we get as impulse balance
 the so-called Langevin equation, an It$\hat{\rm o}$ stochastic differential equation of the form
\begin{equation*} %
\dd X(t) = b(X(t),t) \dt +   \sqrt{\frac{2k_B\mu_T}{\zeta}}       \dd W(t), \quad t \geq 0. \end{equation*}
%As initial condition, we set $X(0)=0$.\footnote{\color{red} I included the initial condition here (not yet done for Y) however setting the initial condition for X to zero results in the delta for psi) ??? and thus there would be no freedom in setting $psi_0$ anymore ???
%but I might be completely wrong - on the other hand not fixing X to zero at initial time is also no way out}
%Considering the integral form and introducing the time-average, 
The partial differential equation describing the evolution of the probability density function 
%
%first moment
%(i.e. the mathematical expectation) 
$\tilde\psi:=\tilde\psi(x_1,x_2,t)$
%:=\langle X(t) \rangle$ 
%
of the random variable $X(t)$ is the standard Fokker--Planck equation \cite{pavliotis2014stochastic}:
\begin{equation} \label{eq:stFP}
\pt \tilde\psi= \sum_{i=1}^2 \left( -\div_{x_i}\!\big( b_i((x_1^{\rm{T}},x_2^{\rm {T}})^{\rm T} 
,t)\, \tilde\psi \big)  + \frac{k_B\mu_T}{\zeta} \Delta_{x_i} \tilde\psi \right),
\end{equation}
with $x_1$ and $x_2$ considered to be column-vectors in $\mathbb{R}^3$.
%, see also \cite[Equation (11)]{sokolov2006field} and \cite[Proof of Theorem 1]{magdziarz2009stochastic} for a computation of the moments to the time-fractional Fokker--Planck equation.

At this point, we deviate from the usual derivation of the Fokker--Planck equation for the evolution of the probability density function of the stochastic process $X$ 
%(see, \cite[Section 5.3.2]{pavliotis2014stochastic}), 
and introduce a subordination of the Langevin equation. This allows us to model trapping events to the motion of the particles. For $\alpha \in (0,1)$,  let $U_\alpha (\cdot)$ be an $\alpha$-dependent subordinator %for which the Laplace transform of its probability density function is 
with expectation %Laplace transform 
$\mathbb{E}[e^{-\lambda U_\alpha(\tau)}]= \text{exp}( -\tau \Phi_\alpha(\lambda) )$, where
\begin{equation*}
\Phi_\alpha(\lambda ) := \tau_0^{\alpha -1} \lambda^\alpha
\end{equation*}
is the so-called Laplace exponent and $\tau_0$ is a characteristic time-scale (to be fixed). We note that the limiting value of $\alpha =1 $ results in the standard integer-order case. The inverse subordinator $S_\alpha ^t $, defined as the first-passage time of $U_\alpha$, is then given by 
\begin{equation*}
S_\alpha^t := \inf_{\tau > 0} \{ \tau \,:\,  U_\alpha(\tau) > t \} .
\end{equation*}
%We keep the notation and denote from now on the new  parent process by $X$.
%It is assumed to satisfy
%
Suppose that $Y_\alpha(t)$ is a solution of the stochastic differential equation
\[ \mathrm{d}Y_\alpha(t) =  b(Y_\alpha(t),U_\alpha(t))\mathrm{d}t +   \sqrt{\frac{2k_B\mu_T}{\zeta}}  \mathrm{d} W(t),\quad t \geq 0.
\]
Define $X(t):=Y_\alpha(S_\alpha^t)$. It then follows from Theorem 1 in \cite{Magd2014} that if $b$ is twice continuously differentiable with respect to its
variables and satisfies the usual Lipschitz condition, then the probability density function of the process $X$ is a solution of the time-fractional
Fokker--Planck equation
%
%It follows that the the subordinated 
%Langevin equation 
%\begin{equation} \label{Eq:Langevin} \dd X(t) = b\big(X(U_\alpha(t)),U_\alpha(t)\big) \dt +  \sqrt{\frac{2k_B\mu_T}{\zeta}}   \dd W(t). 
%\end{equation} 
%By integrating the subordinated Langevin equation \cref{Eq:Langevin} on the time interval $(0,t)$, the subordinated process $Y_\alpha(t):=X(S_\alpha^t)$ can be written as
%\begin{equation} \label{eq:sub}
%Y_\alpha(t)= \int_0^t b(X(\tau),\tau) \dd S_\alpha^\tau +   \sqrt{\frac{2k_B\mu_T}{\zeta}}  W(S_\alpha^t),
%\end{equation}
%where the Lebesgue--Stieltjes integral of the vector function $b(\cdot,\cdot)$ is to be understood componentwise. 
%Equation \eqref{eq:sub} is a direct consequence of \eqref{Eq:Langevin} and the fact that $U_\alpha (S_\alpha^t) =t$.
%Since the Fourier transform of $Y_\alpha(t)$ is holomorphic in a neighborhood of zero (see, for example, \cite[Proof of Theorem 1]%{magdziarz2009stochastic}), the moments determine the distribution in a unique way; see \cite[Section VII.3]{feller}.
%Compared with the standard Fokker--Planck equation, we have to consider the inverse Laplace transform of the symbol of $\tau_0^{1-\alpha} \lambda^{1-%\alpha}$, which results in the presence of the differential operator $\tau_0^{1-\alpha} \ptb$.
% We recall that the limiting case of
%$\alpha =1$ corresponds to the inverse Laplace transform of $1$ and thus to the operator $\delta(t)$. As in the standard case, we define the first moment, still denoted by $\tilde\psi$, by $ \tilde \psi =\tilde\psi(x_1,x_2,t):=\langle Y_\alpha(t) \rangle$.
%These preliminary observations 
resulting from the replacement of $\tilde \psi$ on the right-hand side  of \eqref{eq:stFP} by $ \tau_0^{1 -\alpha} \ptb \tilde \psi$, i.e., 
$$
\pt \tilde\psi= \sum_{i=1}^2 \left( -\div_{x_i}\!\big( b_i((x_1^{\rm{T}},x_2^{\rm {T}})^{\rm T} 
,t)\,\tau_0^{1 -\alpha} \ptb \tilde\psi \big)  + \frac{k_B\mu_T}{\zeta} \Delta_{x_i}  \tau_0^{1 -\alpha} \ptb \tilde\psi \right).$$

Next, we perform the linear change of variables $(x_1,x_2) \mapsto (\frac{1}{2}(x_1 + x_2), x_2 - x_1)=:(x,q)$, whereby we have identified a point $x \in \Omega \subset \mathbb{R}^3$ in the (macroscopic) flow domain $\Omega$ with the center of mass of the dumbbell, and have, as before, denoted by $q$ the vector pointing from $x_1$ to $x_2$. 
 By recalling the definition \cref{Def:Vectorb} of $b(\cdot,\cdot)$ and setting $\psi(x,q,t):=\tilde\psi(x-\tfrac12 q,x+\tfrac12 q,t)$, we find that
\begin{equation} \label{Eq:Langevin3} 
\begin{aligned}
 \pt \psi \, + \,  & \tau_0^{1 -\alpha} \div_x \bigg(\frac{u(x-\tfrac12q,t)+u(x+\tfrac12q,t)}{2} \ptb\psi \bigg)\\ &\quad+ \tau_0^{1 -\alpha} \div_q\bigg(\big(u(x+\tfrac12q,t)-u(x-\tfrac12q,t)\big) \ptb\psi - \frac{2F(q)}{\zeta}  \ptb\psi \bigg) \\ &= \frac{k_B\mu_T \tau_0^{1 -\alpha}}{2\zeta} \Delta_x \ptb\psi + \frac{2k_B\mu_T \tau_0^{1 -\alpha}}{\zeta} \Delta_q \ptb\psi. 
\end{aligned} 
\end{equation}
To proceed, we assume `local homogeneity', i.e., that the spatial variation of the velocity field over the microscopic length-scale of a single dumbbell is negligibly small. Consequently, the arithmetic mean $\big(u(x-\frac12q,t)+u(x+\frac12q,t)\big)/2$ can be approximated by $u(x,t)$ in the second term on the left-hand side of \cref{Eq:Langevin3}.  In the case of the third term, we use Taylor expansion of $u$ about the point $x$ to obtain 
\begin{equation} \begin{aligned} u(x+\tfrac12 q,t)-u(x-\tfrac12 q,t) &= \nablax u(x,t) q + \mathcal{O}(|q|^3) \\ &=\Big(\sigma\big(u(x,t)\big) + \omega\big(u(x,t)\big)  \Big)q + \mathcal{O}(|q|^3),
		\end{aligned}
	\label{Eq:ApproxTaylor}
	\end{equation}
where we have further split the gradient of $u$ into its symmetric and antisymmetic parts, which are, respectively, defined as follows:
 \begin{equation} \label{Eq:Omega} \sigma(u):=\frac{\nablax u + (\nablax u)^{\mathrm{T}}}{2}, \qquad \omega(u):=\frac{\nablax u - (\nablax u)^{\mathrm{T}}}{2}.
 	\end{equation}
In the approximation \cref{Eq:ApproxTaylor}, we omit the $\mathcal{O}(|q|^3)$ term and further, we also omit the symmetric part of the gradient of $u$, i.e., we consider the, so called, \textit{corotational model}. 
While the omission of the $\mathcal{O}(|q|^3)$ term from \cref{Eq:ApproxTaylor} can be justified on the grounds that $|q| \ll 1$, our omission of the term $\sigma(u(x,t))$ from the additive decomposition $\nabla_x u(x,t)=\sigma(u(x,t))+\omega(u(x,t))$ is for purely technical reasons.
%
%and has no physical justification: unlike the case of $\alpha =1$, with our current mathematical tools we are unable to address the question of %existence of global weak solutions to the general noncorotational Navier--Stokes--Fokker--Planck system in the time-fractional case, where the full %gradient $\nabla_x u$ features in equation \cref{Eq:DerivFP} below instead of the skew-symmetric gradient $\omega(u)$. 
In the time-fractional corotational model considered here, polymer molecules are therefore allowed to rotate, but they are forced to do so without stretching. This modelling assumption weakens the coupling between the Navier--Stokes equation and the (time-fractional) Fokker--Planck equation; for example, if the initial datum for $\psi$ happens to be spherically symmetric with respect to $q$ and independent of the spatial variable $x$, then the Fokker--Planck equation is decoupled from the Navier--Stokes equation. In this respect, the model that we study here is no
different from the corotational model considered (in the case of $\alpha=1$) in the works  \cite{lions2000global,lions2007global,schonbek2009existence,masmoudi2008well,barrett2005existence,barrett2009numerical,debiec2023corotational,lin2008global,busuioc2014fene}.
Thus, we consider the following corotational Fokker--Planck equation with a time-fractional derivative:
\begin{equation} \label{Eq:FokkerDim1}\begin{aligned} 
&\pt \psi + \tau_0^{1- \alpha}
\div_x(u \ptb\psi) + \tau_0^{1- \alpha} \div_q\big(\omega( u) q \ptb\psi \big) \\ 
&\quad = \frac{k_B \mu_T \tau_0^{1- \alpha}}{2\zeta} \Delta_x \ptb\psi + \frac{2k_B\mu_T \tau_0^{1- \alpha}}{\zeta} \Delta_q \ptb\psi+ \tau_0^{1- \alpha} \div_q\Big( \frac{2F(q)}{\zeta}  \ptb\psi\Big).
\end{aligned} 
\end{equation}
We shall rely on the fact that $q^{\mathrm{T}} \omega(u) q \equiv 0$ thanks to the skew-symmetry of $\omega(u)$. In the general noncorotational case $q^{\mathrm{T}} (\nabla_x u)q$ is of course not identically nonzero in the flow domain. For $\alpha=1$ at least, the proof of the existence of large-data global-in-time weak solutions to the general noncorotational Navier--Stokes--Fokker--Planck system relies, instead, on an entropy estimate (cf.  \cite{barrett2011existence}). It is this entropy estimate that needs to be replicated in the time-fractional case, for $\alpha \in (0,1)$, as a key
ingredient of the proof of the existence of global-in-time large-data weak solutions. The analysis of the  general noncorotational time-fractional model is deferred to future work.

%Because the continuum mechanical “macroscopic” equations of incompressible fluid flow are coupled to a “microscopic” model, the polymer model under consideration is a microscopic–macroscopic model. Here, the microscopic equation is the time-fractional Fokker--Planck equation
%\eqref{Eq:DerivFP}, describing the statistical properties of polymer molecules in the continuum. We begin by presenting these equations and collecting the relevant assumptions on the various parameters featuring in the model.

%Let $\Omega \subset \R^3$ be a Lipschitz domain and $D \subset \R^3$ a bounded open ball (centered at the origin) of admissible elongation vectors $q$. Let $u:\Omega  \times [0,T) \to \R^3$ be the velocity field and $\psi:\Omega \times D \times [0,T) \to \R$ the probability density function that represents the probability at time $t$ of finding the center of mass of a dumbbell in the volume element $x+\d x$ with the end-point of its elongation vector within the volume element $q+\dq$. 
%\\[7cm]

  Let $\psi(x,q,t) $ denote from now on the probability density function for a collection of $N\gg 1$ dumbbells. As we are dealing with a dilute polymeric fluid, the polymer molecules suspended in the fluid are assumed not to interact with each other and they move without self-interaction. The function $\psi$ therefore satisfies the same partial differential equation, \eqref{Eq:FokkerDim1}, as in the case of a single dumbbell. The only difference is in the choice of the initial datum $\psi^0 \geq 0$ for the Fokker--Planck equation. 
A typical choice of $\psi^0$ in the present context is $$\psi^0(x,q) = \frac{1}{N}\sum_
{j=1}^N \alpha_j(q) \Psi_j(x),$$ where $\alpha_j(q) \geq 0$ for all $q \in D$ and $\int_D \alpha_j(q)\d q = 1$, $j=1,\ldots,N$, and $\Psi_j \geq 0$, for all $x \in \Omega$ and $\int_\Omega \Psi_j (x) \d x =1$, $j=1,\ldots,N$. For example, one may choose $\Psi_j$  as a mollifier (i.e. a nonnegative $C^\infty_0$ approximation to the Dirac measure) concentrated at a point $z_j \in \Omega$, $j=1,\ldots,N$, with $\{\Psi_j\}_{j=1}^N$ forming a scaled partition of  unity; here $z_j$ can be thought of as the initial location of the center of mass of the $j$-th dumbbell.

For the sake of simplicity we shall confine our attention to the case when the boundary $\partial\Omega$ of the macroscopic flow domain $\Omega \subset \mathbb{R}^3$ has no inflow or outflow parts. The polymeric fluid under consideration is therefore confined to $\Omega$, and its velocity will be supposed to satisfy the no-slip boundary condition $u(x,t)=0$ for all $(x,t) \in \partial\Omega \times (0,T)$. 
Thus the
total number $N\gg 1$ of polymer molecules contained in $\Omega$ remains constant in time. The number density 
$\rho_P := \rho_P(x,t)$, in $[m^{-3}]$, of the polymer molecules contained in $\Omega$ is called the \textit{polymer number density} and it is related to $N$  by $ N = \int_\Omega \rho_P(x,t) \d x$. The polymer number density is further related to the probability density function $\psi$ satisfying the Fokker--Planck equation  \eqref{Eq:FokkerDim1} by 
\begin{align}\label{eq:rho-psi} 
\rho_P(x,t) = N \int_D \psi(x,q,t) \d q,
\end{align}
with the understanding that $\psi$ has the SI unit of $[m^{-6}]$. We shall supplement the Fokker--Planck equation \eqref{Eq:FokkerDim1} with no-flux (homogeneous Neumann) boundary conditions on $\partial\Omega \times D \times (0,T)$ and on $\Omega \times \partial D \times (0,T)$, which will ensure that $\int_{\Omega}\int_D \psi(x,q,t) \d q \d x$ is constant in time, and is therefore equal to $$\int_{\Omega}\int_D \psi(x,q,0) \d q \d x =  \int_{\Omega}\int_D \psi^0(x,q) \d q \d x=1.$$

Having established the “microscopic” equation that describes the statistical properties of polymer molecules in the continuum, we turn our attention to the continuum mechanical “macroscopic” equations of motion of the incompressible fluid in which the polymer molecules are suspended. We note that polymeric fluids are non-Newtonian fluids and the presence of the $N \gg 1$ polymer molecules contributes an additional term, $\tau=\tau(\psi)$, to the stress tensor appearing in the balance of linear momentum equation in the Navier--Stokes system: a symmetric polymeric extra stress tensor, which we shall define below. We assume that the evolution of the  velocity field $u$ and the pressure $p$ is governed by the incompressible Navier--Stokes system
\begin{equation} \label{Def:NS} \begin{aligned}
		\rho(\pt u + (u \cdot \nablax) u )- \eta \Delta_x u + \nablax p &= \div_x \tau (\psi) &&\quad\text{in }\Omega\times (0,T), \\
		\div_x u &=0 &&\quad\text{in }\Omega\times (0,T),\end{aligned} \end{equation}
supplemented with the homogeneous Dirichlet boundary condition $u=0$ on $\partial \Omega \times (0,T)$ and the initial condition $u(0)=u^0$ in $\Omega$
at $t=0$. 
As before, $\eta$ denotes the dynamic viscosity and $\rho$ stands for the macroscopic density, which is assumed to be constant in space and time.   The polymeric extra-stress tensor $\tau=\tau(\psi)$ is defined by the so-called \textit{Kramers expression}, see, e.g., \cite{bird1987dynamics2},
\begin{equation} 
	\label{Def:taud} \tau(\psi):=  \rho_P  \big(\C(\psi)- k_B \mu_T I_3\big),
\end{equation}
where, as before, $\rho_P$ is the polymer number density, $k_B$ and $\mu_T$ are the Boltzmann constant and the absolute temperature, respectively, and $I_3$ is the $3\times 3$ identity matrix. 
Finally,  the $3\times 3$ symmetric  tensor $\C$ appearing in the expression for  $\tau$ is defined by 
\begin{equation} 
\label{Def:C} \C(\psi):= \frac{\int_D  F(q) q^{\mathrm{T}} \psi \d  q }{ \int_D   \psi \d  q }  
,
\end{equation}
see, e.g., \cite{suli2018mckeanvlasov}. Because the polymer number density $\rho_P$ is related to $\psi$ by \eqref{eq:rho-psi}, the polymeric extra stress tensor is simplified to 
\[ \tau(\psi)  =\tau^1(\psi)+ \tau^2(\psi),\]
where
\begin{align} \label{eq:tau1}
 \tau^1(\psi) &:=   N \int_D  F(q) q^{\mathrm{T}} \psi \d  q , \\ \label{eq:tau2}
 \tau^2(\psi) &:=       - N k_B \mu_T  \int_D \psi \d q \ I_3.
\end{align}
Recall that $\psi$ is assumed to have the SI unit $[m^{-6}]$.

We note that the definition of $\C(\psi)$  results in the SI unit $[ kg/(s^2 m^2)] $ for the polymeric extra stress tensor $\tau=\tau(\psi)$.  
 The tensor $\tau$ is responsible for coupling the velocity $u$ and the pressure $p$ to the probability density function  $\psi$. 
Dividing the Navier--Stokes momentum equation by the macroscopic density $\varrho$ and introducing the kinematic viscosity as $\nu := \eta/\rho$, we arrive at the Navier--Stokes system in its dimensional form:
\begin{equation} \label{Eq:NavierDim}
\begin{aligned}
	\pt u + (u \cdot \nablax) u - \nu \Delta_x u + \frac 1\rho\nablax p &= \frac 1\rho\div_x \tau(\psi) &&\quad\text{in }\Omega\times (0,T), \\
	\div_x u &=0 &&\quad\text{in }\Omega\times (0,T).
\end{aligned} 
\end{equation}
In the next section we shall perform a nondimensionalization of the Navier--Stokes--Fokker--Planck system. In order to distinguish a nondimensionalized
quantity from its dimensional form, we shall use the subscript $_{\mathrm{dim}}$ for dimensional quantities in cases where confusion might arise; so, for example, we shall write $\tau_{\text{dim}}$ to indicate the original dimensional form of the polymeric extra stress tensor, while $\tau$ will henceforth signify its nondimensionalized form. Similarly $\mathcal{C}_{\mathrm{dim}}$ will denote the original dimensional form of $\mathcal{C}$ (cf. \eqref{Def:C}), while $\mathcal{C}$ will signify its form following nondimensionalization.

\subsection{Nondimensionalization} \label{se:non}
Next, we transform the Navier--Stokes--Fokker--Planck system, see \eqref{Eq:NavierDim} and  \eqref{Eq:FokkerDim1}, into its nondimensionalized form. To this end, we define the quantities $\hat x$, $\hat t$ and $\hat u$ by setting
$$x=L_0 \hat x, \quad t=T_0 \hat t, \quad u(x,t)=U_0 \hat u(\hat x,\hat t), $$
where	$L_0$ and $T_0$ stand for  the characteristic macroscopic length-scale and time-scale, respectively, and $U_0$ denotes the characteristic velocity of the macroscopic flow. 
 In a similar manner, we introduce $\hat q$ by letting $q=l_0 \hat q $, where $l_0$ is a characteristic microscopic length-scale, recall that $\psi$ was assumed to have the SI unit $[m^{-6}]$, and we define the dimensionless quantity $\hat\psi$ by 
\begin{align} \label{eq:hatpsi}
 \hat \psi(\hat x, \hat q, \hat t) := (L_0 l_0)^{3}\, \psi(x,q,t) . 
\end{align}
 We point out that scaling $\psi$ by a constant does not affect the definition of the $\C_{\text{dim}}(\psi)$ given by \eqref{Def:C}; i.e., for any positive constant $a$ we have $\C_{\text{dim}}(a \psi)  =\C_{\text{dim}}(\psi)$. 
Scaling $\psi$ by a positive constant, however, affects its relationship to the polymer number density $\rho_P (x,t)$. To avoid the nondimensionalization of $\rho_P (x,t)$, we do not consider the  definition \eqref{Def:taud} of $ \tau_{\text{dim}}(\psi) $ but of its equivalent form given by \eqref{eq:tau1} and \eqref{eq:tau2}.

%For ease of readability, we omit the hats in the arguments and in the partial differential symbols but keep the ones in the velocity and the probability density function.
 Consequently, we obtain from the time-fractional Fokker--Planck equation \eqref{Eq:FokkerDim1} the following partial differential equation
\begin{align} \label{Eq:FokkerNondim1}
& \nonumber \frac{1 }{T_0}\partial_{\hat t} \hat \psi + \frac{\tau_0^{1- \alpha} U_0 }{T_0^{1-\alpha} L_0} \div_{\hat{x}}( \hat u (\psi) \partial_{\hat t}^{1-\alpha} \hat \psi) + \frac{ \tau_0^{1- \alpha} U_0}{T_0^{1-\alpha} L_0} \div_{\hat q} \big(\omega( \hat u (\psi) ) \hat{q}\, \partial_{\hat t}^{1-\alpha} \hat \psi\big)  \\ 
&\quad =    \frac{k_B \mu_T \tau_0^{1- \alpha}}{2\zeta T_0^{1-\alpha} L_0^{2}}  \Delta_{\hat x} \,\partial_{\hat t}^{1-\alpha} \hat \psi +  \frac{2k_B\mu_T \tau_0^{1- \alpha}}{\zeta l_0^2 T_0^{1-\alpha}} \Delta_{\hat q} \, \partial_{\hat t}^{1-\alpha} \hat \psi +  \frac{2 H \tau_0^{1- \alpha}}{\zeta T_0^{1-\alpha}} \div_{\hat q}( \hat F(\hat q)  \partial_{\hat t}^{1-\alpha} \hat \psi) 
\end{align}
defined on the nondimensionalized domain $\widehat \Omega \times \widehat D \times (0, \hat T)$. Here we have used  the notation $\hat u(\psi) $ to indicate that the velocity field is understood to depend on the original (dimensional) probability density function $\psi$, and will only be expressed as a function depending on the nondimensionalized probability density function $\hat \psi$ in the next step. 
We set the characteristic macroscopic and microscopic time-scales to, respectively, $T_0:=L_0/U_0$, $\tau_0 := T_0$  and the nondimensionalized force in case of the standard FENE type model to
\begin{align} \label{eq:hatF}
\hat F(\hat q) ;= \frac{\hat q}{1 - |\hat q|^2/|\hat q_{\max} |^2},  \qquad \hat q_{\max} := q_{\max}/l_0.
\end{align}
The  prefactors of the last two terms on the right-hand side of \eqref{Eq:FokkerNondim1} are equal if the microscopic length-scale $\ell_0$ is defined as follows: 
\begin{align} \label{Eq:Def:l_0}
\ell_0^2:= \frac{k_B \mu_T}{H}.\end{align}
Next we introduce the Deborah number $\lambda$, defined as the ratio of the time it takes for the material to adjust to applied stresses/deformations and the characteristic time scale, $T_0$, by
\[ \lambda := \frac{\zeta/(4H)}{T_0} = \frac{\zeta}{4H T_0}  = \frac{\zeta \ell_0^2}{4 k_B \mu_T T_0}.\]
Finally, we define the nondimensional parameter 
\[ \varepsilon := \frac{k_B \mu_T}{2\zeta  U_0 L_0},\] 
whose numerator and denominator both carry the SI unit of $\mathrm{J = [kg\,m^2/s^2]}$ corresponding to energy. Thus, we obtain from \eqref{Eq:FokkerNondim1} by multiplying it with $T_0$ the following nondimensional equation:
\begin{align} \label{Eq:DerivFP}
\nonumber 	\partial_{\hat t} \hat \psi + (\hat u (\psi) \cdot \nabla_{\hat x}) \partial_{\hat t}^{1 - \alpha} \hat \psi + \div_{\hat q}\,\big(\omega(\hat u (\psi)) {\hat q}\, \partial_{\hat t}^{1- \alpha} \hat \psi\big) \, \qquad \qquad \qquad \\  = \varepsilon \Delta_{\hat x}\, \partial_{\hat t}^{1- \alpha} \hat \psi + \frac{1}{2\lambda} \div_{\hat q}\,(\nabla_{\hat q} \,\partial_{\hat t}^{1 - \alpha} \hat \psi + \hat F(\hat q)\, \partial_{\hat t}^{1- \alpha} \hat \psi).
\end{align} 
 We note in passing that the ratio of the diffusion coefficients $\varepsilon$ and $1/(2\lambda)$ featuring in the equation \eqref{Eq:DerivFP} is equal to $(\ell_0/(2L_0))^2$, which is $\ll 1$, so the first-term on the right-hand side of \eqref{Eq:DerivFP}, called the center-of-mass diffusion term, is frequently neglected in practical considerations. Crucially, we shall retain this term in the equation and will continue to work with a strictly positive center-of-mass diffusion coefficient $\varepsilon$, as is implied by the derivation of the Fokker--Planck equation \eqref{Eq:DerivFP} performed above.

In the same manner, we multiply the Navier--Stokes system \eqref{Eq:NavierDim} by $T_0/U_0$ and arrive at the following nondimensionalized system posed on $\hat \Omega \times (0,\hat T)$:
\begin{equation} \label{Eq:NavierNonDim}
\begin{aligned}
	\partial_{\hat t} \hat u + (\hat u \cdot \nabla_{\hat x}) \hat u - \frac{1}{\text{Re}} \Delta_{\hat x} \hat u + \nabla_{\hat x} \hat p &= \div_{\hat x} \hat \tau(\hat \psi), \\  %&&\quad\text{in }\Omega\times (0,T), \\
	\div_{\hat x} \hat u &=0,% &&\quad\text{in }\Omega\times (0,T), 
\end{aligned} 
\end{equation}
where $\text{Re}$ stands for the Reynolds number and $\hat p$ for a scaled pressure. Defining $\tau(\hat \psi )$ in a proper way allows us to replace $\hat u(\psi)$ in our nondimensionalized Fokker--Planck equation by $\hat u(\hat \psi)$ or, in compact notation, by $\hat u$.
Thanks to the reformulation of the stress, we reconsider the two terms $\hat \tau^1(\hat \psi)$ and $\hat \tau^2 (\hat \psi)$, appearing in $\hat \tau (\hat \psi) := \hat \tau^1(\hat \psi) + \hat \tau^2 (\hat \psi)$, separately. The definition \eqref{eq:tau1} of $\tau^1 (\psi)$ in combination with the nondimensionalization steps and the definition of $\hat F$ and $\hat \psi$, see \eqref{eq:hatF} and \eqref{eq:hatpsi}, respectively yields
\begin{align} \label{eq:tau1nd}
\frac{1}{\rho} \frac{T_0}{U_0} \frac{1}{L_0} H l_0^2 l_0^3  &\int_{\hat D} \hat F(\hat q) \hat q^{\mathrm T} \hat \psi \d \hat q\, (L_0 l_0)^{-3}  N  \\ & \nonumber 
= \ \frac{N}{\rho L_0^3}  \frac{k_B \mu_T}{U_0^2} \int_{\hat D} \hat F(\hat q) \hat q^{\mathrm T} \hat \psi(\hat q) \d \hat q  =  \gamma\, \hat \C(\hat \psi) =: \hat \tau^1(\hat \psi),
\end{align}
with the dimensionless quantities $\gamma$ and $\hat \C(\hat \psi)$ being defined by
\begin{equation*}
	\gamma := \frac{ k_B \mu_T N }{\rho U_0^2 L_0^3}  \quad\mbox{and}\quad  \hat \C(\hat \psi):= \int_{\hat D} \hat F(\hat q) \hat q^{\mathrm T} \hat \psi \d \hat q.
\end{equation*}
To define $\hat \tau^2 (\hat \psi) $, we start from the scaled macroscopic momentum balance equation and the definition \eqref{eq:tau2}:
\begin{align} \label{eq:tau2nd}
\frac{1}{\rho} \frac{T_0}{U_0} \frac{1}{L_0} (- k_b \mu_T) l_0^3 & \int_{\hat D} \hat \psi \d \hat q \, (L_0 l_0)^{-3}  N   \ I_3  \\ & \nonumber
= \ -\frac{N}{\rho L_0^3} \frac{k_B \mu_T}{U_0^2}  \int_{\hat D} \hat \psi \d \hat q  \ I_3
 =  \gamma    \int_{\hat D} \hat \psi \d \hat q  \ I_3  =: \hat \tau^2(\hat \psi).
\end{align}
For the sake of notational simplicity we drop from now on all $\, \hat \cdot \,$ symbols.

We close this section by recalling from \eqref{Eq:NavierNonDim} and \eqref{Eq:DerivFP} our nondimensionalized model problem:
\begin{equation} \label{Def:FP}  \begin{aligned}
	\pt \psi + (u \cdot \nablax) \ptb \psi + \div_q\big(\omega(u) &q \ptb \psi\big) &&  \\  - \tfrac{1}{2\lambda} \div_q(\nablaq \ptb  \psi +  \ptb  \psi)  &=\eps \ptb \Delta_x \psi  &&\quad\text{in } \Omega \times D \times (0,T),  \\
\pt u + (u \cdot \nablax) u - \frac{1}{\text{Re}} \Delta_x u + \nablax p &= \div_x \tau(\psi) &&\quad\text{in }\Omega\times (0,T), \\
\div_x u &=0 &&\quad\text{in }\Omega\times (0,T),
\end{aligned} \end{equation}
and we supplement this system of equations with the following boundary conditions:
$$
\begin{aligned}
\left(\tfrac{1}{2\lambda} (\nablaq \ptb \psi + U'q \ptb\psi)-\omega(u) q \ptb \psi \right) \cdot n_{\partial D} &=0 \qquad \text{on } \Omega \times \partial D \times (0,T), \\
\eps \nablax \ptb \psi \cdot n_{\partial \Omega} &=0 \qquad \text{on } \partial\Omega \times  D \times (0,T), \\
	u  &= 0 \qquad \text{on } \partial\Omega  \times (0,T).
\end{aligned}$$
The nondimensional initial data $u^0(x)$ and $\psi^0(x,q)$ for the velocity field and the probability density function, respectively, are obtained from their dimensional counterparts. We note that the scaling \eqref{eq:hatpsi} ensures
that $\int_\Omega \int_D \psi^0(x,q) \d x \d q = 1$; this then guarantees, thanks to the homogeneous Neumann boundary conditions on $\psi$, that $\int_\Omega \int_D \psi(x,q,t) \d x \d q = 1$ for all $t \geq 0$.

	\section{Mathematical preliminaries} \label{Sec:Prelim}
In this section, we introduce some useful definitions and results regarding the fractional derivative in the sense of Riemann--Liouville and recall the Aubin--Lions lemma, which is a key result featuring in proofs of existence of weak solutions to nonlinear PDEs based on compactness arguments. %

For a Hilbert space $H$ with inner product $(\cdot,\cdot)_H$ and norm $\|\cdot\|_H$, we shall denote the duality pairing between $H$ and its dual space $H'$ by $\langle \cdot,\cdot\rangle_H$. We shall denote the inner product on the Bochner space $L^2(0,T;H)$ by $(\cdot,\cdot)_{L^2H}$, and we shall write $(\cdot,\cdot)_{L^2_tH}$ when in this inner product the temporal interval of integration is $(0,t)$ for some $t \in (0,T)$ rather than the complete interval $(0,T)$, i.e., $$(u,v)_{L^2_tH}:=\int_0^t (u(s),v(s))_H \, \text{d}s \qquad \forall\, u,v \in L^2(0,T;H).$$
The norm induced by this inner product will be denoted by $\|\cdot\|_{L^2_t H}$.

\subsection{Riemann--Liouville kernels}
The Riemann--Liouville kernel function $g_\alpha$ of order $\alpha$ is defined by $g_\alpha(t):=t^{\alpha-1}/\Gamma(\alpha)$, $t \in (0,T)$, for $\alpha > 0$ and $g_0(t):=\delta_0(t)$ (the Dirac distribution concentrated at $0$) for $\alpha=0$. We observe that  $g_\alpha \in L^p(0,T)$ for any $\alpha\in (1-1/p,1)$ and $p \in [1,\infty)$, and the kernel function satisfies the following semigroup property; see \cite[Theorem 2.4]{diethelm2010analysis}:
\begin{equation} \label{Eq:Semigroup}
	\ga * g_\beta = g_{\alpha+\beta} \qquad \forall\, \alpha,\beta \geq 0.
\end{equation} 
%This can be proved as follows: One applies Fubini's theorem to interchange the order of integration 
%$$\begin{aligned}(\ga * g_\beta *u)(t) &=\frac{1}{\Gamma(\alpha)\Gamma(\beta)} \int_0^t (t-s)^{\alpha-1} \int_0^s (s-\tau)^{\beta-1} u(\tau) \dd \tau \dd s  \\
%	&=\frac{1}{\Gamma(\alpha)\Gamma(\beta)} \int_0^t u(\tau) \int_\tau^t (t-s)^{\alpha-1} (s-\tau)^{\beta-1}   \dd s \dd \tau,
%\end{aligned}$$
%and the substitution $s=\tau+\sigma(t-\tau)$ then yields
%$$\begin{aligned}(\ga * g_\beta *u)(t) 
%	&=\frac{1}{\Gamma(\alpha)\Gamma(\beta)} \int_0^t u(\tau) (t-\tau)^{\alpha+\beta-1} \int_0^1 (1-\sigma)^{\alpha-1} \sigma^{\beta}   \dd \sigma \dd \tau.
%\end{aligned}$$
%Lastly, we observe using the fundamental property of the Gamma function that the second integral is equal to $\Gamma(\alpha)\Gamma(\beta)/\Gamma(\alpha+\beta)$, see \cite[Theorem D.6]{diethelm2010analysis}, from which we deduce the desired semigroup property \cref{Eq:Semigroup} of $g_\alpha$.

We note that when $\alpha \in (0,1)$, one can bound the $L^p(0,t)$-norm of a function $u:(0,T) \to \R$ by its convolution with $\ga$ as follows: for any $t \in (0,T]$, we have that
\begin{equation}\begin{aligned} \|u\|_{L^p(0,t)}^p := \int_0^t |u(s)|^p \ds  &\leq t^{1-\alpha} \int_0^t (t-s)^{\alpha-1} |u(s)|^p \ds \\ &\leq T^{1-\alpha} \Gamma(\alpha) \big(\ga * |u|^p\big)(t).	
\end{aligned} 
\label{Eq:KernelNorm}
\end{equation}
This implies that the space $$L^p_\alpha(0,T):=\big\{u:(0,T) \to \R:\sup_{t \in (0,T)} (\ga*|u|^p)(t) < \infty \big\},$$ is indeed a subspace of $L^p(0,T)$.
%Further, this estimate can be generalized for a nonnegative function $u:(0,T) \to \R_{\geq 0}$ and for $0<\beta<\alpha<1$ in the following way:
%$$(\ga * u)(t)=\frac{1}{\Gamma(\alpha)} \int_0^t (t-s)^{\beta-1} \frac{(t-s)^{\alpha-1}}{(t-s)^{\beta-1}} u(s) \ds \leq \frac{T^{\alpha-\beta}\Gamma(\beta)}{\Gamma(\alpha)} (g_\beta * u)(t).$$
If the order $\alpha$ of the kernel function $g_\alpha$  is larger than 1, then one can exploit the semigroup property of the kernel and apply Young's convolution inequality (cf. Lemma 3.2 in \cite{Oparnica}) as follows:
$$(g_{1+\alpha}*u)(t)=(g_1*\ga*u)(t)=\int_0^t (\ga*u)(s) \ds \leq \|\ga\|_{L^1(0,t)} \|u\|_{L^1(0,t)},$$
for any $u \in L^1(0,T)$ and any $t \in (0,T]$.

\subsection{Time-fractional derivative} 
We can rewrite the definition of the Riemann--Liouville derivative stated in \cref{Eq:RL} in a compact form by using the convolution operator $*$ as $\pta w=\pt (\gb * w)$.
We refer to the classical textbooks \cite{diethelm2010analysis,baleanu2012fractional} and the newer monographs \cite{jin2021fractional,chen2022fractional} regarding fractional calculus and fractional differential equations.

 We define the fractional Riemann--Liouville--Bochner space   for $\alpha \in (0,1)$ and $p \in [1,\infty)$ on $(0,T)$ with values in $H$ by $$\W^{\alpha,p}(0,T;H):=\big\{u \in L^p(0,T;H) : \gb * u \in W^{1,p}(0,T;H)\big\}.$$
Here, the convolution $\ast$ is of course understood to be with respect to the temporal variable $t \in (0,T)$. In the limit, 
when $\alpha=1$, we have that $g_{1-\alpha} = g_0=\delta$, and then $$\W^{1,p}(0,T;H):=W^{1,p}(0,T;H):=\big\{u \in L^p(0,T;H) : \pt u \in L^p(0,T;H)\big\}.$$
However for $0 < \alpha < 1$,
the Riemann--Liouville space $\W^{\alpha,p}(0,T;H)$ differs from the fractional-order Sobolev--Bochner space
$$W^{\alpha,p}(0,T;H):=\Big\{u \in L^p(0,T;H) : (s,t) \mapsto \tfrac{\|u(t)-u(s)\|_H}{|t-s|^{\alpha+1/p}} \in L^{p}((0,T)\times(0,T))\Big\},$$
which can be confirmed by noting that the function $g_\alpha$ is an element of $\W^{\alpha,p}(0,T):=\W^{\alpha,p}(0,T;\R)$ for $\alpha \in (1-\tfrac{1}{p},1)$ but not of $W^{\alpha,p}(0,T)$; see \cite[Proposition 3.13]{carbotti2021note}. 
%Therefore, we are not able to apply classical results such as embedding theorems for Sobolev--Bochner spaces.

\begin{remark}  
Even though the space $\W^{\alpha,p}(0,T)$ is not a subspace of the Sobolev--Slobodecki\u{\i} space $W^{\alpha,p}(0,T)$, it is nevertheless continuously embedded into $C([0,T])$, the space of uniformly continuous functions defined on $[0,T]$, for $\alpha \in (1- \frac{1}{p},1]$ and $p \in [1,\infty)$; see, \cite[Remark 6.2]{carbotti2021note}. 
\end{remark}

\begin{comment}
Therefore, small values of $\alpha$ have to be studied carefully.
Further, we observe that $\ga$ does not belong to $L^p(0,T)$ for $\alpha \in (0,1-\tfrac{1}{p}]$ (e.g., $\ga \notin L^2(0,T)$ for $\alpha \in (0,\frac12]$) and therefore, we find that $$(\gb*\phi)(0)=0 \qquad \forall\, \phi \in \W^{\alpha,p}(0,T;H),~ \alpha \in (0,1-\tfrac{1}{p}],$$
by the inverse convolution property  \cref{Eq:InverseConvolution}. However, this might contradict a given nontrivial initial condition $\phi^0$. E.g., for $\phi \in \W^{\alpha,2}(0,T;H):=\W^{\alpha,2}(0,T;H)$ it has to hold that $(\gb*\phi)(0)=0$ for $\alpha\in (0,\tfrac12]$ and therefore, PDE solutions with this regularity are only well-posed for $\phi^0=0$. Such an issue can be avoided by studying PDEs of the form $\pta(\phi-\phi^0)=f(\phi)$ and considering instead the regularity of $\phi-\phi^0$, i.e., $\phi \in \W^{\alpha,2}_{\phi^0}(0,T;H)$. However, the time-fractional model \cref{Eq:System} in this work is not of this translated form and therefore, we cannot expect that this system is well-posed for non-zero initials. We note that $\psi_0=0$ is physically unreasonable anyway for probability density functions and, moreover, we will naturally observe in the existence's proof below that the restriction $\alpha \geq \frac12$ naturally appears in the energy estimates.
\end{comment}

We also introduce the following Riemann--Liouville space incorporating a homogeneous initial condition at $t=0$, albeit in a somewhat nonstandard manner:
$$\begin{aligned}
\W^{\alpha,p}_{0}(0,T;H)&:=\big\{u \in \W^{\alpha,p}(0,T;H) : (\gb*u)(0)=0 \big\}.
%\W^{\alpha,p}_{u^0}(0,T;H)&:=\big\{u \in L^p(0,T;H) : u-u^0 \in \W_0^{\alpha,p}(0,T;H) \big\}.
\end{aligned}$$ 
We note that the function $\gb*u:[0,T] \to H$ has a well-defined trace at $t=0$ (even when the function $u$ itself might not have one) thanks to the continuous embedding
$$\gb*u \in W^{1,p}(0,T;H) \hookrightarrow AC([0,T];H).$$
For a given element $z \in H$, the convolution $\gb*z$ should be understood to mean the function $t \mapsto (\gb*g_1)(t) z \in H$; recall that $g_1(t)\equiv 1$ for all $t\geq 0$. Thus, $z \in H$ is in this context now thought of as the mapping $t \mapsto g_1(t)z \in \mathcal{W}^{\alpha,p}(0,T;H)$,  for $\alpha \in (0,1)$, $p \in [1,\infty)$ and $0<\alpha p < 1$, or if $\alpha =1$ and $p \in [1,\infty)$. 
Thanks to the semigroup property \eqref{Eq:Semigroup} we then have that $$t\mapsto (\gb*z)(t)=z\,g_{2-\alpha}(t)=\frac{z}{\Gamma(2-\alpha)}  t^{1-\alpha} \in C([0,T];H)$$ for any $\alpha \in [0,1]$. Thus, for $\alpha \in (0,1]$, $p \in [1,\infty)$ and $z \in H$ we define the following `translated' Riemann--Liouville space:
\begin{equation} \label{Eq:RLSpaceU0}
    \W^{\alpha,p}_{z}(0,T;H):=\big\{u \in L^p(0,T;H) : u-z \in \W^{\alpha,p}(0,T;H), ~ (g_{1-\alpha}*u)(0)=0 \big\}.
\end{equation}
Note that if $\alpha \in (0,1)$, $p \in [1,\infty)$ and $0<\alpha p<1$, or if $\alpha=1$ and $p \in [1,\infty)$, then $u-z \in \W^{\alpha,p}(0,T;H)$ if, and only if $u \in \W^{\alpha,p}_0(0,T;H)$, and therefore, for such $\alpha$ and $p$ we have that $\W^{\alpha,p}_{z}(0,T;H) = \W^{\alpha,p}_0(0,T;H)$ irrespective of the choice of $z \in H$.

Next, we state an inverse convolution (or deconvolution) property. Its name stems from the fact that convolution with the kernel $\ga$ acts as an inverse mapping on the operator of taking $\alpha$-th fractional derivative, up to a term that involves the initial value at $t=0$.
\begin{lemma}[Inverse convolution] Let $\alpha \in (0,1]$ and $p\in [1,\infty)$. Suppose further that $H$ is a Hilbert space and $z \in H$. Then, for any $t \in (0,T)$, we have the following equalities:
\begin{align} 	\label{Eq:InverseConvolution1}
(\ga * \pta u)(t) &= u(t) - (\gb*u)(0)\ga(t) \quad &&\forall\, u \in \W^{\alpha,p}(0,T;H), \\ 	\label{Eq:InverseConvolution}
		(\ga* \pta u)(t)    &=u(t)  &&\forall\, u \in \W_{z}^{\alpha,p}(0,T;H). \end{align} 
\end{lemma}
\begin{proof}
	We start with the proof of the equality \cref{Eq:InverseConvolution1}.
Recall that for any function $u \in \W^{\alpha,p}(0,T;H)$ we have $\gb*u \in  AC([0,T];H)$, and the fundamental theorem of calculus for absolutely continuous functions therefore yields, for any $t\in [0,T]$,
$$(\gb *u)(t) - (\gb*u)(0) = \int_0^t \partial_s (\gb * u)(s) \ds=(g_1 * \pta u)(t).$$
We convolve this equality with the kernel $\ga$ and make use of the semigroup property \cref{Eq:Semigroup} to obtain
$$(g_1*u)(t) - (\gb*u)(0) g_{1+\alpha}(t)=g_{1+\alpha}*\pta u,$$
where we have used that  $g_\alpha*1=g_\alpha*g_1=g_{1+\alpha}$, because $\alpha \Gamma(\alpha) = \Gamma(1+\alpha)$.
Next, we differentiate this equality in $t$ and observe that $\pt (g_1*u)=u$, $\pt g_{1+\alpha}=g_\alpha$, and $\pt (g_{1+\alpha}*v)=g_\alpha*v$, which yields \cref{Eq:InverseConvolution1}.
We finally note that  \cref{Eq:InverseConvolution} follows trivially from \cref{Eq:InverseConvolution1} and \cref{Eq:RLSpaceU0}.
%We consider an element $u \in \W_{u^0}^{\alpha,p}(0,T;H)$, i.e., there exists an element $v \in \W_{0}^{\alpha,p}(0,T;H)$ with $u-u^0=v$ and using $v$ in \cref{Eq:InverseConvolution1}, we obtain $\ga*\pta v=v$, i.e.,
%\begin{equation*} \begin{aligned}\ga* \pta (u-u^0)   &= u-u^0 &&\forall\, u \in \W_{u^0}^{\alpha,p}(0,T;H).
	%\end{aligned} \end{equation*} 
%Moreover, we can split the left-hand side thanks to the linearity of the fractional derivative and obtain 
 %\begin{equation*} \begin{aligned}
	%	\ga* \pta u    &= u-u^0  + \ga * \pta u^0 =u  &&\forall\, u \in \W_{u^0}^{\alpha,p}(0,T;H),
	%\end{aligned} \end{equation*} 
	%where we have used that   $\ga * \pta 1 = \ga*\gb =1$ thanks to the semigroup property  \cref{Eq:Semigroup}.
 \end{proof}
 
The following result is a direct consequence of the interaction between fractional derivatives and kernel functions.
\begin{corollary} The following identities hold:
\begin{equation} \label{Eq:DerivativeofKernel}  \begin{aligned}
	\pta (\ga * u ) &=\pt ( \gb * \ga * u) = \pt (1*u) = u &&\forall\, u \in L^1(0,T;H), \\
	\ptb \pta u &=  \pt (\ga * \pta u) = \pt u &&\forall\, u \in  W_{0}^{1,1}(0,T;H).
\end{aligned}
\end{equation}
\end{corollary}

%However, in our setting of the time-fractional Navier--Stokes--Fokker--Planck system, we have already seen that the Riemann--Liouville derivative appears on the left-hand side without the translation of an initial value. This already explains intuitively the restriction on the values of $\alpha$ in the theorem of the system's well-posedness, see \cref{Thm:WellPosedness} below.

%As in the integer-order setting, there are continuous and compact embedding results for Riemann--Liouville spaces.
We shall require the following special case of the classical Aubin--Lions lemma; see \cite{simon1986compact}. Suppose that the Hilbert spaces $V,H,Z$ form a Gelfand triple $V \com H \con Z$. then, the following classical compact embeddings hold: 
\begin{equation} \begin{aligned} \label{Eq:aubin} 
W^{1,1}(0,T;Z) \cap L^p(0,T;V) &\com L^{p}(0,T;H), &&p \in [1,\infty), \\
W^{1,r}(0,T;Z) \cap L^\infty(0,T;V) &\com C([0,T];H), && r \in (1,\infty);
\end{aligned}\end{equation} 
see \cite{simon1986compact}. Several fractional counterparts of the Aubin--Lions lemma have been proposed; see \cite{ouedjedi2019galerkin,wittbold2020bounded,li2018some}. We make use of  the following result; see \cite[Corollary 3.2]{ouedjedi2019galerkin}:
\begin{equation*} \begin{aligned} %
\W^{\alpha,1}(0,T;Z) \cap L^p(0,T;V) &\com L^r(0,T;H), &&p \in (1,\infty), \quad r \in [1,p), \quad \alpha \in (0,1).
\end{aligned}\end{equation*} 
The proof can be easily adapted to the limit case $r=p$ if the $\alpha$-th fractional derivative is in a better space than $L^1(0,T;Z)$. This is done for Caputo derivatives in \cite{li2018some}. In fact, we obtain
%$r \in (\frac{p}{1+\alpha p},\infty) \cap [1,\infty)$.  In the spacial case when $p=2$ and $\alpha \in (\frac12,1]$ it yields
\begin{equation} \begin{aligned} \label{Eq:aubinfractional2} %
		\W^{\alpha,r}(0,T;Z) \cap L^p(0,T;V) &\com L^p(0,T;H), &&r\in (1,\infty), \quad \alpha \in (0,1).
\end{aligned}\end{equation} 

\begin{comment}
Next, we require a Gronwall-type inequality that allows an additional nonnegative factor $b \in L^1(0,T)$ in the integrand on the right-hand side of the inequality. Particularly, this function is only assumed to be integrable, and it is allowed to degenerate.
\begin{lemma}[Gronwall, cf. {\cite[Lemma II.4.10]{boyer2012mathematical}}] \label{Lem:Gron4}
    Let $C_1,C_2$ be  nonnegative constants and let $b\in L^1(0,T)$ be nonnegative. If the  function $u \in L^\infty(0,T)$  satisfies the inequality
    $$u(t) \leq C_1+C_2 \int_0^t b(s) u(s) \, \textup{d}s \qquad \text{for a.a. } t \in (0,T], $$
    then 
    $$u(t) \leq C_1 \textup{exp}\Big(C_2\int_0^t  b(s) \, \textup{d}s\Big) \qquad \text{for a.a. } t \in (0,T]. $$
\end{lemma}
\end{comment}

%\subsection{Fractional chain inequality}
The classical chain rule does not hold for fractional derivatives, but one can use the following inequality as a remedy; see \cite[Theorem 2.1]{vergara2008lyapunov}:
\begin{equation} \label{Eq:ChainOriginal}  \frac12 \pta \|u\|^2_H +\frac12 \gb(t) \|u\|_H^2 \leq (u,\pta u)_H \quad \forall\, u \in \W_{z}^{\alpha,2}(0,T;H),
\end{equation}
for $z \in H$ and almost all $t \in (0,T)$.

\section{Model revisited} \label{Sec:Form}

Having summarised the required results from fractional calculus, we revisit the mathematical model that we have derived in \Cref{Sec:Derivation}.
Let us assume for the moment that the solution $\psi$ to the Fokker--Planck equation belongs to $\mathcal{W}^{1-\alpha,p}(0,T;H) \cap C([0,T];H)$ for some $\alpha \in (0,1)$ and a suitable Hilbert space $H$, to be chosen. As $\psi \in C([0,T];H)$, it follows that $\|(g_\alpha \ast \psi)(t)\|_H \leq \frac{t^\alpha}{\Gamma(1+\alpha)}\|\psi\|_{C([0,T];H)}$, and therefore  $(g_\alpha \ast \psi)(0)=0$. Hence, $\psi \in \mathcal{W}^{1-\alpha,p}_0(0,T;H)$. It then follows from \eqref{Eq:InverseConvolution1}, with $\alpha$ replaced by $1-\alpha$ and $u=\psi$ there, that $(g_{1-\alpha}\ast \partial_t^{1-\alpha} \psi)(t) = \psi(t)$ for $t \in (0,T)$. 
Motivated by these properties, we introduce the auxiliary function $\phi$ by
\begin{equation} \label{Eq:Substitute} \phi := \ptb \psi = \pt (\ga * \psi),
\end{equation} 
whereby $\psi=g_{1-\alpha}*\phi$. We then have that $\pt \psi = \pt(g_{1-\alpha} \ast \phi) = \pta \phi$; and, thanks to the assumed continuity of $\psi$ (i.e. $\psi \in C([0,T];H)$) it
makes sense to require attainment of the initial condition $\psi(0) = \psi^0$, i.e. $(g_{1-\alpha} \ast \phi)(0) = \psi^0$. We shall therefore introduce the substitution $\phi:=\partial_t^{1-\alpha} \psi$ in \eqref{Def:FP}, which results in the following system of equations:
\begin{equation} \begin{aligned}
	\pt u + (u \cdot \nablax) u - \nu \Delta_x u + \nablax p - \div_x \tau( \gb * \phi) &=0, \\
	\div_x u &=0, \\
	\pta \phi  + (u \cdot \nablax) \phi + \div_q (\omega(u) q \phi)-\tfrac{1}{2\lambda} \div_q(\nablaq \phi + U'q  \phi)   -\eps  \Delta_x  \phi&=0, \end{aligned} 
\label{Eq:System}
\end{equation}
subject to the initial conditions $u(0)=u^0$ and $(\gb*\phi)(0)=\psi^0$ for a given nonnegative $\psi^0$ that fulfils $\int_D \psi^0 \dq=1$. Furthermore, we equip the system with the following boundary conditions:
\begin{align}\label{eq:neumannbc}
\begin{aligned}
u &=0 \qquad\text{on } \partial \Omega \times (0,T), \\
\left(\tfrac{1}{2\lambda} (\nablaq \phi + U'q \phi)-\omega(u) q \phi \right) \cdot n_{\partial D} &=0 \qquad\text{on } \Omega \times \partial D \times (0,T), \\
\eps \nablax \phi \cdot n_{\partial \Omega} &=0 \qquad\text{on } \partial\Omega \times  D \times (0,T).
\end{aligned}
\end{align}

\subsection{The Maxwellian and Maxwellian-weighted function spaces} \label{Sec:Maxwell}
We introduce the normalized Maxwellian by 
\begin{equation}
	\label{Def:Max} M(q)=\frac{e^{-U(\tfrac12 |q|^2)}}{\int_D e^{-U(\tfrac12 |s|^2)} \dd s}.
\end{equation}
Moreover, we define the (Maxwellian-weighted) Hilbert spaces
$$\begin{alignedat}{5}
	&\mathcal{H}=\{h \in L^2(\Omega;\R^d): \div \, h =0\}, \qquad ~\mathcal{H}_0&&=\{h \in \mathcal{H} : h \cdot n_{\partial \Omega} = 0 \text{ on } \partial \Omega\}, 
	\\ &\mathcal{V}=\{v \in H^1(\Omega;\R^d) : \div \, v = 0\}, \qquad ~ \mathcal{V}_0&&=\{v \in \mathcal{V} : v|_{\partial \Omega} = 0 \text{ on } \partial \Omega\},
	\\  &\mathcal{Y}=L^2(\Omega \times D),  \qquad\quad\,\,\,   \widehat{\mathcal{Y}}= L^2_M(\Omega \times D) &&=\{y \in \mathcal{Y}: \|y\|_{\widehat{\mathcal{Y}}} := \|M^{1/2}y \|_\mathcal{Y}<\infty \}, 
	\\ &\mathcal{X} = H^1(\Omega \times D), \qquad\quad \widehat{\mathcal{X}}= H_M^1(\Omega \times D) &&= \{\phi \in \mathcal{X}: \|\phi\|_\hX <\infty  \},  \\
&\mathcal{Z}=H^1(D;H^1(\Omega)), ~ \widehat{\mathcal{Z}} =H_M^1(D; H^1(\Omega))&&=\{\zeta \in \mathcal{Z}: \|\zeta\|_\hZ <\infty  \}, 
\end{alignedat}$$
where the norms on $\hX$ and $\hZ$ are defined by $\|\phi\|_\hX^2:= \|\phi\|_\hY^2 + \|\nablaq \phi\|_\hY^2 + \|\nablax \phi\|_\hY^2$ and $\|\zeta\|_\hZ^2 :=\|\zeta\|_\hX^2 + \|\nablax \nablaq \zeta\|_\hY^2$. Obviously, $H_M^2(\Omega \times D) \subseteq \hZ$, where $H^2_M(\Omega \times D)$ is the subspace of $H^1_M(\Omega \times D)$ consisting of all functions defined on $\Omega \times D$ whose second (weak) partial derivatives belong to $\hY=L^2_M(\Omega \times D)$.
We refer to \cite{barrett2005existence} regarding theoretical results on these weighted Hilbert spaces. In particular, we have the Gelfand triples
$$\begin{aligned} &\HSV \com \HS \hookrightarrow \HSV', \quad &&\HSV_0 \com \HS_0 \hookrightarrow \HSV_0', \\
	&\mathcal{X} \com \mathcal{Y} \hookrightarrow \mathcal{X}', \quad &&\hX \com \hY \hookrightarrow \hX',
\end{aligned}$$
where $\mathcal{V}'$, $\mathcal{V}'_0$, $\mathcal{X}'$ and $\hX'$ denote the dual space of, respectively, 
$\mathcal{V}$, $\mathcal{V}_0$, $\mathcal{X}$ and $\hX$. 

Using the definition of the normalized Maxwellian $M$, see \cref{Def:Max}, we have that $$M(q)\nablaq M(q)^{-1}=-M(q)^{-1} \nablaq M(q) =\nablaq U(\tfrac12 |q|^2) = U'(\tfrac12 |q|^2) q.$$ We introduce the scaled variable $\hphi=\phi/M$ and  with the formula 
$$M \nablaq \hphi = \nablaq \phi + M \nablaq M^{-1} \phi  = \nablaq \phi + U'q \phi$$
we can rewrite the fractional Fokker--Planck equation in \cref{Eq:System} as
$$ \pta \phi + (u \cdot \nablax) \phi + \div_q \big(\omega(u) q \phi\big) = \tfrac{1}{2\lambda} \div_q(M \nablaq \phim) + \eps \Delta_x \phi.$$
As was indicated earlier, we shall confine ourselves here to considering the corotational model, i.e.,
$\omega(v)=-\omega(v)^{\mathrm{T}}$, $q^{\mathrm{T}} \omega(v) q = 0$; if $\div\, v = 0$ it then follows that 
\begin{equation}\label{Eq:SigmaZero} 
\big(M \hphi \,\omega(v)q,\nablaq \hphi\big)_{\mathcal{Y}} = \frac12 \big(M \omega(v)q, \nablaq \hphi^2\big)_{\mathcal{Y}}=-\frac12 \big(\div_q(M \omega(v) q),\hphi^2\big)_{\mathcal{Y}} = 0; \end{equation}
see \cite{barrett2009numerical,barrett2005existence}.
We note in passing that partial integration yields the following equalities: 
\begin{equation} \label{Eq:IntParts} \begin{aligned}
		-2\big(M \omega(u) q \hat\varphi,\nablaq \hphi\big)_{\mathcal{Y}} &=  \big(\nablax(M\hat\varphi \nablaq \hphi)q,u\big)_{\mathcal{Y}} + \big(u\cdot q, \div_x (M\hat\varphi\nablaq \hphi)\big)_{\mathcal{Y}}
		\\ &= \big(M \nablax \hat\varphi (\nablaq \hphi)^{\mathrm{T}} q,u\big)_{\mathcal{Y}} + \big(M \hat\varphi \nablax \nablaq \hphi\, q,u\big)_{\mathcal{Y}} \\ &\quad + \big(u \cdot q,M \nablax \hat\varphi \cdot \nablaq \hphi\big)_{\mathcal{Y}} + \big(u \cdot q,M \hat\varphi \,\div_x \nablaq \hphi\big)_{\mathcal{Y}}.
\end{aligned} \end{equation} 

We recall that the stress tensor $\tau(\psi) = \tau_1(\psi) + \tau_2(\psi)$ is of the form  given by \eqref{eq:tau1nd} and \eqref{eq:tau2nd}; i.e., 
%%%%%%%%%%%%%
%
\begin{align}\label{Eq:tau1tau2}\tau^1(\psi)=\gamma\, \C(\psi),\quad \tau^2(\psi)= \gamma \int_{D} \psi \d q  \ I_3,\quad  \C(\psi):= \int_{D} F(q) q^{\mathrm T} \psi \d q,
\end{align}
where $\gamma>0$ is a dimensionless constant.

%%%%%%%%%%%%
For $\C(M \hpsi)$, we are in a setting that allows us to deduce the following bound; see also \cite[Eq. (3.7)]{barrett2009numerical}:
\begin{equation} \label{Eq:C}  \begin{aligned}
		\int_\Omega |\C(M \hpsi)|^2 \d x & =
		\int_\Omega \left| \, \int_{D} F(q)  q^{\mathrm T} M \hpsi \d q \, \right|^2 \d x
		\\ &\leq \int_D M  | F(q)  q^{\mathrm T} |^2   \dq \, \int_{\Omega \times D} M|\hpsi|^2 \d(x,q)  \\ & \leq C  \| \hpsi\|_{\hY}^2 \quad \forall\, \hpsi \in \hY. \end{aligned}\end{equation}

	\section{Existence of weak solutions} \label{Sec:Analysis}
In this section, we prove the existence of a weak solution to the time-fractional Navier--Stokes--Fokker--Planck system. We use a Galerkin procedure and discretize the partial differential equations in space and derive suitable energy estimates. We will emphasize the places where the time-fractional derivative comes into play. We shall then pass to the limit in the sequence of Galerkin approximations to deduce the existence of a weak solution. We shall proceed step by step and prove this result through several lemmas. While, for physical reasons, the derivation of the system was in the previous sections discussed in the case of $d=3$ space dimensions, the analysis below applies to both two and three space dimensions. We begin by introducing the concept of a weak solution to the time-fractional corotational Navier--Stokes--Fokker--Planck system under consideration. %\medskip

\begin{definition} \label{Eq:DefWeak} Suppose that $d \in \{2,3\}$.
We call the pair $(u,\hphi)$ a weak solution to the system \cref{Eq:System}, \cref{eq:neumannbc} provided that   
\begin{align*}
u &\in L^\infty(0,T;\mathcal{H}_0) \cap L^2(0,T;\mathcal{V}_0) \cap W^{1,\tfrac{8}{4+d}}(0,T;\mathcal{V}_0'), \\ 
%L^{
\hphi &\in L^2(0,T;\hX), \quad \pta \hphi \in  L^{\tfrac{8}{4+d}}\big(0,T; \hZ' \big),
\end{align*}
satisfies the initial conditions $u(0)=u^0$, $(\gb*\hphi)(0)=\hpsi^0 := \psi^0/M$ and  the variational problems
\begin{align}
		\langle \pt u,v \rangle_{L^{8/(4-d)}\mathcal{V}}  + \big((u \cdot \nablax)u,v\big)_{L^2\mathcal{H}} + \nu (\nablax u,\nablax v)_{L^2\mathcal{H}} && \label{Eq:NS} \\[-.1cm] + k_B \mu_T \big(\C(M \gb *\hphi),\nablax v\big)_{L^2\mathcal{H}} =0, && \notag  \\[.1cm]
		\langle  \pta \hphi,\hat\zeta\rangle_{L^{8/(4-d)} \hZ}  - (u \hphi,\nablax \hat\zeta)_{L^2\hY}+ \frac{1}{2\lambda} (\nablaq \hphi,\nablaq \hat\zeta)_{L^2\hY}  + \eps (\nablax \hphi,\nablax \hat\zeta)_{L^2\hY} && \label{Eq:FP} \\[-.2cm]   + \frac12 (\nablax(\hphi\nablaq \hzeta)q,u)_{L^2\hY} - \frac12 \big(u\cdot q,\div_x (\hphi\nablaq \hzeta)\big)_{L^2\hY} = 0. &&  \notag
\end{align}
for all $v \in L^{8/(4-d)}(0,T;\mathcal{V}_0)$ and $\hat\zeta \in L^{8/(4-d)}(0,T;\hZ)$. In these variational problems and hereafter, for a Hilbert space $H$ and $p \in [1,\infty]$, subscripts of the form $L^pH$ and $L^p_tH$ appearing in the various inner products, norms, and duality pairings, signify $L^p(0,T;H)$ and $L^p(0,t;H)$, respectively. %Similarly, for a reflexive Banach space $B$, $\langle \cdot, \cdot \rangle_{C B}$ denotes the duality pairing between $C([0,T];B)$ and its dual space, 
%$\mathcal{M}(0,T;B')$ the space of all $B'$-valued finitely additive finite signed measures defined on $(0,T)$, which are absolutely continuous with respect to the Lebesgues measure, equipped with the total variation norm over $(0,T)$.
\end{definition} %\medskip

We summarize the assumptions that we require for proving the existence of a weak solution in the sense of \cref{Eq:DefWeak} in Assumption \ref{Ass:WellPosedness} below. %\medskip

\begin{assumption} \label{Ass:WellPosedness}
	Let the following assumptions hold:
	\begin{itemize}
		\item $D \subset \R^d$,  $d\in\{2,3\}$, is a bounded open ball centered at the origin, $\Omega \subseteq \R^d$ is a Lipschitz domain (i.e., bounded, open, connected set in $\R^d$, with a Lipschitz-continuous boundary $\partial\Omega$), and $T<\infty$ is a fixed final time;
  \item $u^0 \in \mathcal{H}_0$, $\hpsi^0 \in \hX$ with $\hpsi^0 \in H^1_M\big(D; H^{1+d/2+\delta}(\Omega)\big)$ for $\delta>0$ arbitrarily small;
  \item $\tau(\psi) =\tau_1(\psi) + \tau_2(\psi)$ is given by \eqref{Eq:tau1tau2}, with the identity matrix $I_3 \in \mathbb{R}^{3 \times 3}$ replaced by the identity matrix $I_d \in \mathbb{R}^{d \times d}$ in the definition of $\tau_2(\psi)$,  and $\C$ satisfies \cref{Eq:C};
    \item $\alpha \in (1/2,1)$;
  \item $k_B,\mu_T, \nu,\lambda,\eps >0$.
	\end{itemize}
\end{assumption} %\medskip

The main result of this paper is the following theorem asserting the existence of global-in-time large-data weak solutions to the time-fractional Navier--Stokes--Fokker--Planck system under consideration. 

%\medskip

\begin{theorem} \label{Thm:WellPosedness}
	Let \cref{Ass:WellPosedness} hold; 
	then, there exists a weak solution $(u,\hphi)$ to the system \cref{Eq:System}, \cref{eq:neumannbc} in the sense of \cref{Eq:DefWeak}.
\end{theorem}%\medskip

In order to prove this theorem, we state several lemmas, which will eventually imply \cref{Thm:WellPosedness}. We begin by constructing a sequence of Galerkin approximations $\{(u_k,\hphi_k)\}_{k=1}^\infty$ to the system of partial differential equations under consideration, resulting in a system of fractional-order ordinary differential equations, which admits a local-in-time solution $(u_k,\hphi_k)$
for each $k \geq 1$ thanks to  standard theory. We then derive an energy estimate for the sequence of Galerkin approximations, which is uniform with respect to $k$; this then implies that, for each $k \geq 1$, the local-in-time solution of the fractional-order system of ordinary differential equations can be extended to the entire time-interval $[0,T]$; it also implies the existence of a weakly/weakly-$*$ convergent subsequence $(u_{k_j},\hphi_{k_j})$. Finally, we pass to the limit $j \to \infty$ and apply a compactness argument to deduce that the limiting pair of functions 
$(u,\hphi)$ is in fact a weak solution to the system of partial differential equations in the sense of \cref{Eq:DefWeak}. The Galerkin method has been applied to various time-fractional PDEs; see, e.g., \cite{fritz2020time,fritz2021sub,fritz2021equivalence,vergara2015optimal}; it has also been applied to Navier--Stokes--Fokker--Planck systems in \cite{knezevic2009heterogeneous,barrett2009numerical,barrett2011finite,barrett2012finite}, with an integer-order Fokker--Planck equation. 

\subsection{Galerkin discretization}
We follow the construction of \cite[Section 2.1]{bulicek2013existence} and conclude by the Hilbert--Schmidt theorem \cite[Lemma A.4]{bulicek2013existence} the existence of a countable set $\{h_j\}_{j=1}^\infty$ of eigenfunctions in $\mathcal{V}_0 \cap H^{1+\frac{d}{2}+\delta}(\Omega)^d$, with $\delta>0$ arbitrarily small, whose linear span is dense in $\mathcal{H}_0$ such that the $h_j$, $j\in \{1,2,\dots\}$, are orthonormal in $\mathcal{H}$ and orthogonal in $H^{1+\frac{d}{2}+\delta}(\Omega)^d$ in the sense that $(h_j,h_i)_{H^{1+\frac{d}{2}+\delta}(\Omega)}=\lambda_j \delta_{i,j}$ for any $i,j \in \{1,2,\ldots\}$ and $\lambda_j>0$ for all $j=1,2,\ldots$. Similarly, we fix a countable set $\{y_j\}_{j=1}^\infty$ in $H^2_M(\Omega \times D)$ that forms an orthogonal system in $\hX$ and an orthonormal system in $\hY$. 
%Consider the high-order elliptic problem of finding a solution tuple $(u,\lambda) \in H^{1+d}(\Omega)^d \times \R$ of
%$$(u,v)_{H^{1+d}(\Omega)} + (\nabla u,\nabla v)_{L^2(\Omega)}  = \lambda (u,v)_{\mathcal{H}} \qquad \forall  v \in W^{1+d}(\Omega)^d.$$
%By a version of the Hilbert--Schmidt theorem, see \cite[Lemma A.4]{bulicek2013existence}, there is a countable set of eigenfunctions $\{h_j\}_{j=1}^\infty$, which are orthogonal in the inner product of $H^{1+d}(\Omega)^d$ and orthonormal in the inner product of $L^2(\Omega)^d$. We note that $H^{1+d}(\Omega)^d$ is continuously embedded into $W^{1,\infty}(\Omega)$
We then define the $k$-dimensional linear spaces
\begin{align*}
	\mathcal{H}_k  :=\text{span}\{ h_1,\dots,h_k\}, \quad 
	\widehat{\mathcal{Y}}_k  :=\text{span}\{ y_1,\dots,y_k\},
	%	Z_K &=\text{span}\{ z_1,\dots,z_k\},
\end{align*}
and we consider the Galerkin approximations
\begin{equation}\begin{gathered}
		u_k (t) = \sum_{j=1}^k u_k^j(t) h_j,
		\quad \hphi_k (t) = \sum_{j=1}^k \hphi^j_k(t) y_j,
	\end{gathered}
	\label{Eq:GalerkinAnsatzFunctions}
\end{equation}
where 
$u^{j}_k$ and $\hphi^j_k$ are real-valued coefficient functions for all $j \in \{1,\dots,k\}$.
 %Let $h>0$ denote a discretization parameter tending to zero. As in \cite{barrett2009numerical}, we choose finite-dimensional spaces $\hY_k^x \subset W^{1,\infty}(\Omega)$ and $\hY_k^q \subset W^{1,\infty}(D)$ such that it holds
%$$\text{dist}_{W^{1,\infty}(\Omega)} (\eta,\hY_k^x) \to 0, \qquad \text{dist}_{W^{1,\infty}(D)} (\xi,\hY_k^q) \to 0,$$
%as $h\to 0$ for all $\eta \in C^\infty(\overline\Omega)$ and $\xi \in C^\infty(\overline D)$. Moreover, we define the tensor space $\hY_k=\hY_k^x \otimes \hY_k^q \subset W^{1,\infty}(\Omega \times D)$ and note that $\widehat{X}_h \subset X \subset \hX$. 
%Further, we define the finite-dimensional spaces $W_h$, $R_h$ and $\mathcal{H}_k$ such that
%$$\begin{aligned}W_h &\subset \mathcal{H}_0^1(\Omega;\R^d) \cap W^{1,\infty}(\Omega;\R^d), \quad R_h \subset L_0^2(\Omega), \\ \mathcal{H}_k&=\{w_h\in W_h: (\div_x w_h,r_h)_{\mathcal{H}} \,\forall r_h \in R_h \},\end{aligned}$$
%where $\cup_{h>0} W_h$ and $\cup_{h>0} R_h$ are supposed to be dense in $\mathcal{H}_0^1(\Omega;\R^d)$ and $L_0^2(\Omega;\R^d)$, respectively. Further, we assume that for all $v \in V$ there exists a sequence $v_k \in \mathcal{H}_k$ such that $v_k \to v$ in $H^1(\Omega)$ for $h \to 0$. This holds for the typical Galerkin approximation with the availability of an uniform inf-sup condition.
The canonical orthogonal projection onto the finite-dimensional space $\mathcal{H}_k$ is defined by $\Pi_{\mathcal{H}_k}: \mathcal{H} \to \mathcal{H}_k$, $h \mapsto \sum_{j=1}^k (h,h_j)_{\mathcal{H}} h_j$,  and in the same way for $\Pi_{\hY_k}:\hY \to \hY_k$. 
For $h=\sum_{j=1}^\infty (h,h_j)_{\mathcal{H}}h_j$ we have that $$\|h\|^2_{H^{1+\frac{d}{2}+\delta}(\Omega)} = \sum_{j=1}^\infty \lambda_j |(h,h_j)_{\mathcal{H}}|^2, \quad \|\Pi_{\mathcal{H}_k} h\|^2_{H^{1+\frac{d}{2}+\delta}(\Omega)} = \sum_{j=1}^k \lambda_j |(h,h_j)_{\mathcal{H}}|^2,$$
from which we conclude via the Sobolev embedding theorem that, for each $k \geq 1$, 
$$\|\Pi_{\mathcal{H}_k} h\|_{W^{1,\infty}(\Omega)} \leq C\|\Pi_{\mathcal{H}_k} h\|_{H^{1+\frac{d}{2}+\delta}(\Omega)} \leq C\|h\|_{H^{1+\frac{d}{2}+\delta}(\Omega)}.$$

%Thanks to \cite[Theorem 8.1.11]{brenner2008mathematical} and \cite{guzman2009holder}, we have that $\Pi_{\mathcal{H}_k}$ is uniformly $H^1$-stable and $\Pi_{\hY_k}$ is $W^{1,\infty}$-stable.
%Given the initial data $u^0$ and $\psi^0$ from the continuous system, we choose $u_k^0 \in \mathcal{H}_k$ and $\hphi_k^0 \in \hY_k$ such that $u_k^0 = \Pi_{\mathcal{H}_k} u^0$ and $\hpsi_k^0=\Pi_{\hY_k} \hpsi^0$. 

The Galerkin equations read as follows: we wish to find a tuple $(u_k,\hphi_k) \in \mathcal{H}_k \times \hY_k$ for each $k \geq 1$ such that $u_k(0)=u_k^0:= \Pi_{\mathcal{H}_k}u^0$, $(\gb*\hphi_k)(0)=\hpsi_k^0:=\Pi_{\hY_k} \hpsi^0$, and
\begin{align}(\pt u_k,v_k )_{\mathcal{H}}  &+ ((u_k \cdot \nablax)u_k,v_k)_{\mathcal{H}} + \nu (\nablax u_k,\nablax v_k)_{\mathcal{H}}  \label{Eq:NS_dis} \\ 
& + \,k_B \mu_T (\C(M \gb *\hphi_k),\nablax v_k)_{\mathcal{H}} =0, \notag \\
	(\pt (\gb*\hphi_k),\hat\zeta_k)_\hY  &+  
	((u_k \cdot \nablax) \hphi_k, \hzeta_k)_\hY + \frac{1}{2\lambda} (\nablaq \hphi_k,\nablaq \hat\zeta_k)_\hY  \label{Eq:FP_dis} \\ &+ \eps (\nablax \hphi_k,\nabla_x \hat\zeta_k)_\hY- (\omega(u_k) q \hphi_k, \nablaq \hat\zeta_k)_\hY=0, \notag \end{align} for all $v_k \in \mathcal{H}_k$ and $\hat\zeta_k \in \hY_k$.  %\medskip

\begin{lemma} \label{Eq:LemmaExistenceODE}
	Suppose that \cref{Ass:WellPosedness} holds; then, for each $k \geq 1$,
	there exists a local-in-time solution $(u_k,\hphi_k)$ to the Galerkin system \cref{Eq:NS_dis}, \cref{Eq:FP_dis}.
	\end{lemma} %\medskip

\begin{proof}
Let $U_k(t):=(u_k^1(t),\ldots,u_k^k(t))^{\mathrm{T}}$ and $\widehat\Phi_k(t)=(\widehat{\phi}_1^k,\ldots, \widehat{\phi}_k^k)^{\rm T}$. With this notation the Galerkin subsystem \eqref{Eq:NS_dis} becomes an initial-value problem for a system of ordinary differential equations of the form $\ddt U_k  = F(t, U_k, \widehat\Phi_k)$, while, by noting that $\ddt (\gb*\hphi_k)=\pta \hphi_k$, the Galerkin subsystem \eqref{Eq:FP_dis} takes the form of an initial-value problem for a system of fractional-order ordinary differential equations $(\ddt)^\alpha \widehat\Phi_k  = G(t, U_k, \widehat\Phi_k)$. As the functions $F$ and $G$ are continuous with respect to their arguments and locally Lipschitz continuous with respect to their second and third arguments, we can appeal to the generalization of the Cauchy--Lipschitz theorem stated in Theorem 5.1 of \cite{diethelm2010analysis} to deduce the existence of a unique continuous solution, defined on a time interval $[0, T_k]$ where $0<T_k \leq T$, where $U_k$ is, in fact, a continuously differentiable function of $t$ by the classical Cauchy--Lipschitz theorem. 
\end{proof}

\subsubsection{Energy estimates} Next, we derive a $k$-uniform energy estimate, which will allow us to extend, for each $k \geq 1$, the corresponding local-in-time Galerkin solution, whose existence is guaranteed by Lemma \ref{Eq:LemmaExistenceODE}, to the entire time interval $[0,T]$; it will also enable us to extract weakly converging subsequences of Galerkin approximations.  We begin by deriving a bound on the solution to the Galerkin approximation of the Navier--Stokes equation; we shall then derive a bound on the solution to the Galerkin approximation of the Fokker--Planck equation. At the end, we will add the two bounds and apply Gronwall's lemma to obtain a $k$-uniform energy estimate. %\medskip

\begin{lemma} \label{Lem:EstU} Let \cref{Ass:WellPosedness} hold; then the following bound on the Galerkin solution $u_k$, in terms of $u_k^0$ and $\hphi_k$, holds for all $t \in (0,T_k)$:
	\begin{equation} \label{Eq:EnergyU2}
		\frac12 \|u_k(t)\|_{\mathcal{H}}^2  + \frac{\nu}{2} \|\nablax u_k\|_{L^2_t{\mathcal{H}}}^2  \leq \frac12 \|u^0_h\|_{\mathcal{H}}^2 +\frac{ T^{2-2\alpha} k_B^2\mu_T^2}{2\nu (1-\alpha)^2}  \|\hphi_k\|_{L^2_t\hY}^2.
	\end{equation}
\end{lemma} %\medskip

\begin{proof}
We take the test function $v_k=u_k(t)$ in the equation \cref{Eq:NS}, which gives
$$\begin{aligned}
	\frac12 \frac{\dd}{\dd t} \|u_k\|_{\mathcal{H}}^2  + \nu \|\nablax u_k\|_{\mathcal{H}}^2  &= -k_B \mu_T \big(\C(M\gb*\hphi),\nablax u_k\big)_{\mathcal{H}}, 
	\end{aligned}$$
and we can further bound the right-hand side from above to deduce that
$$\frac12 \frac{\dd}{\dd t} \|u_k\|_{\mathcal{H}}^2  + \nu \|\nablax u_k\|_{\mathcal{H}}^2  \leq \frac{\nu}{2} \|\nablax u_k\|_{\mathcal{H}}^2 + \frac{k_B^2\mu_T^2}{2\nu} \|\C(M\gb*\hphi_k)\|_{\mathcal{H}}^2. $$
By bounding the term  $\C(M\gb * \hphi_k)$ as in \cref{Eq:C} we arrive at the inequality
\begin{equation} \begin{aligned}\frac12 \frac{\dd}{\dd t} \|u_k\|_{\mathcal{H}}^2  + \frac{\nu}{2} \|\nablax u_k\|_{\mathcal{H}}^2  &\leq  \frac{k_B^2\mu_T^2}{2\nu} \|\gb*\hphi_k\|_\hY^2. %
\end{aligned} \label{Eq:EnergyU}
\end{equation}
Next, we note that, for any $t \in (0,T_k)$,
\begin{align}\label{Eq:EnergyU1} 
\int_0^t \|(\gb*\hphi_k)(s)\|_\hY^2 \ds \leq \int_0^t  \big(\gb * \|\hphi_k\|_\hY\big)^2(s) \ds\leq \|\gb\|_{L^1_t}^2 \|\hphi_k\|_{L^2_t\hY}^2.
\end{align}
This follows by observing that $g_{1-\alpha}$ only depends on the scalar variable $s$, which permits pulling the $\hY$-norm inside of the convolution, followed by applying Young's convolution inequality (cf. Lemma 3.2 in \cite{Oparnica}) in the resulting integrand. 

We then integrate \eqref{Eq:EnergyU} over the interval $[0,t]$ where $t \in (0,T_k)$ and use \eqref{Eq:EnergyU1} to bound the right-hand side of the resulting inequality. Finally we
note that $g_{1-\alpha}$ is integrable on $(0,t)$ and its integral is bounded by $T^{1-\alpha}/(1-\alpha)$. 
This gives \cref{Eq:EnergyU2}.
\end{proof} %%\medskip

Having derived a bound on $u_k$, we move on to the derivation of a bound on $\hphi^k$ by testing the Galerkin system \eqref{Eq:FP_dis}.  %\medskip

\begin{lemma} \label{Lem:EstPhi} Let \cref{Ass:WellPosedness} hold and let $\gamma>0$ be arbitrary but fixed; then, the following bound on the sequence of Galerkin solutions $\{(u_k,\hphi_k)\}_{k=1}^\infty$ holds:
\begin{equation} \label{Est:SolFP} \begin{aligned} & \frac{\gamma}{2} \big(\gb* \|\hphi_k-\init \ga\|_\hY^2\big)(t) + \frac{\gamma \,T^{-\alpha}}{16\, \Gamma(1-\alpha)} \|\hphi_k\|_{L^2_t\hY}^2  +\frac{\gamma}{2\lambda} \|\nablaq \hphi_k\|_{L^2_t\hY}^2 + \gamma \eps \| \nablax \hphi_k\|_{L^2_t \hY}^2  \\
	&\quad \leq C(\alpha,\gamma) \|M^{1/2} \init\|_{H^1(D;W^{1,\infty}(\Omega)) }^2  \int_0^t g_{2\alpha-1}(s) \|u_k(s)\|_{\mathcal{H}}^2 \ds
	+ C(\alpha,\gamma,T)\| \init\|_\hX^2.
\end{aligned}\end{equation}
\end{lemma} %\medskip

\begin{proof}
We note again that,  thanks to the inverse convolution property, see \cref{Eq:InverseConvolution}, $\hphi_k - \init \ga = \ga * \pta \hphi_k$.
We take this function as the test function in the variational Fokker--Planck equation \cref{Eq:FP}, i.e., $\hat\zeta= \hphi_k - \init \ga = \ga * \pta \hphi_k$, which gives  
\begin{equation} \label{Eq:TestingRHS}\begin{aligned} &(\pta \hphi_k,\ga * \pta \hphi_k)_\hY +\frac{1}{2\lambda}  \|\nablaq \hphi_k\|_\hY^2 + \eps \| \nablax \hphi_k\|_\hY^2\\ %
&\quad =\ga(t) \cdot \Big( \big((u_k \cdot \nablax) \hphi_k, \init\big)_\hY + \frac{1}{2\lambda} \big(\nablaq \hphi_k,\nablaq \init\big)_\hY  \\[0cm] &\qquad + \eps \big(\nablax \hphi_k,\nabla \init\big)_\hY- \big(\omega(u_k) q \hphi_k, \nablaq \init\big)_\hY \Big) =:R.
\end{aligned}\end{equation}
We then use the fractional chain inequality \cref{Eq:ChainOriginal} to bound the left-hand side of \cref{Eq:TestingRHS} from below, which yields
$$
\frac12 \pta \|\hphi_k-\ga \init\|^2_\hY  \leq (\pta \hphi_k,\hphi_k-\ga \init)_\hY = (\pta \hphi_k,\ga * \pta \hphi_k)_\hY. 
$$
Regarding the right-hand side of \cref{Eq:TestingRHS}, we integrate the last term containing $\omega(u_k)$ by parts, see
\cref{Eq:IntParts}, and get
$$\begin{aligned} -  \big(M\omega(u_k)  q \hphi, \nablaq \init\big)_{\mathcal{Y}} &=  \big(M \nablax \hphi_k (\nablaq \init)^{\mathrm{T}} q,u_k\big)_{\mathcal{Y}} + \big(M \hphi_k \nablax \nablaq \init q,u_k\big)_{\mathcal{Y}} \\ &\quad + \big(u_k \cdot q,M \nablax \hphi_k \cdot \nablaq \init\big)_{\mathcal{Y}} + \big(u_k \cdot q,M \hphi_k \div_x \nablaq \init\big)_{\mathcal{Y}}. 
\end{aligned} $$
We apply H\"older's inequality to obtain the following bound on the right-hand side, $R$, of the equality \cref{Eq:TestingRHS}:
$$\begin{aligned} R \leq{}& \ga(t) \cdot\! \Big(  \|u_k\|_{\mathcal{H}}  \|\nablax\hphi_k\|_\hY \|M^{1/2}\init\|_{L^2(D;L^\infty(\Omega)) }\\
%
%[-.2cm]  &+\|u_k\|_{\mathcal{H}}  \|\hphi_k\|_\hY \|M^{1/2}
%\init\|_{L^2(D;W^{1,\infty}(\Omega)) }  \\[0cm]  
%
&+ \frac{1}{2\lambda} \|\nablaq \hphi_k\|_\hY  \|\nablaq \init\|_\hY  + \eps \|\nablax \hphi_k\|_\hY \|\nablax \init\|_\hY   \\
%[-.1cm]  
&+ C\|u_k\|_{\mathcal{H}} \|q\|_{L^\infty(D)}  \|M^{1/2}\nablaq \init\|_{L^2(D;W^{1,\infty}(\Omega)) } \big(\|\hphi_k\|_\hY + \|\nablax \hphi_k\|_\hY\big) \Big).
\end{aligned}$$
Hence, thanks to Young's inequality, we arrive at the following bound on $R$: 
$$\begin{aligned} R \leq&~{} \frac{1}{\eps} \ga(t)^2 \|M^{1/2} \init\|_{L^2(D;W^{1,\infty}(\Omega)) }^2  \|u_k\|_{\mathcal{H}}^2  + \frac{\eps}{4} \|\nablax \hphi_k\|_\hY^2  \\ &{} + \frac{1}{4\lambda} \|\nablaq \hphi_k\|_\hY^2   + \frac{\ga(t)^2}{4\lambda} \|\nablaq \init\|_\hY^2  +\frac{\eps}{4} \|\nablax \hphi_k\|_\hY^2 + \eps \ga(t)^2 \|\nablax \init\|_\hY^2 \\
&{}+\delta  \|\hphi_k\|_\hY + \frac{\eps}{4} \|\nablax \hphi_k\|_\hY^2 + C(\eps,\delta) \ga(t)^2 \|M^{1/2}\nablaq \init\|_{L^2(D;W^{1,\infty}(\Omega)) }^2 \|u_k\|_{\mathcal{H}}^2,
\end{aligned}$$
where $\delta>0$ is sufficiently small, to be chosen appropriately later on. After combining the lower bound on the left-hand side of \cref{Eq:TestingRHS} and the upper bound on the right-hand side we have that
$$\begin{aligned} &\frac12 \pta \|\hphi_k-\ga \init\|^2_\hY +\frac{1}{2\lambda}  \|\nablaq \hphi_k\|_\hY^2 + \eps \| \nablax \hphi_k\|_\hY^2\\ &\quad \leq \frac{1}{\eps}\ga(t)^2 \|M^{1/2} \init\|_{L^2(D;W^{1,\infty}(\Omega)) }^2  \|u_k\|_{\mathcal{H}}^2  + \frac{\eps}{4} \|\nablax \hphi_k\|_\hY^2  \\ &{} \qquad + \frac{1}{4\lambda} \|\nablaq \hphi_k\|_\hY^2   + \frac{\ga(t)^2}{4\lambda} \|\nablaq \init\|_\hY^2  +\frac{\eps}{4} \|\nablax \hphi_k\|_\hY^2 + \eps \ga(t)^2 \|\nablax \init\|_\hY^2 \\
&{}\qquad +\delta  \|\hphi_k\|_\hY^2 + \frac{\eps}{4} \|\nablax \hphi_k\|_\hY^2 + C(\eps,\delta) \ga(t)^2 \|M^{1/2}\nablaq \init\|_{L^2(D;W^{1,\infty}(\Omega)) }^2 \|u_k\|_{\mathcal{H}}^2,\end{aligned}$$
and absorbing terms on the right-hand side into the left-hand side gives
$$\begin{aligned} &\frac12 \pta \|\hphi_k-\ga \init\|^2_\hY +\frac{1}{4\lambda}  \|\nablaq \hphi_k\|_\hY^2 + \frac{\eps}{2} \| \nablax \hphi_k\|_\hY^2\\ &\quad \leq C(\eps,\delta) \ga(t)^2  \|M^{1/2} \init\|_{H^1(D;W^{1,\infty}(\Omega))}   \|u_k\|_{\mathcal{H}}^2  + \delta \|\hphi_k\|_\hY^2 + C(\eps,\delta)\ga(t)^2 \|\init\|_\hX^2.
\end{aligned}$$

We note that $\ga^2=\frac{\Gamma(2\alpha-1)}{\Gamma(\alpha)^2}g_{2\alpha-1}$ is integrable for $\alpha\in (\tfrac12,1)$ and $g_1*g_{2\alpha-1}=g_{2\alpha}$, which is continuous, bounded, and monotonically increasing on $[0,T]$ for $\alpha\in(\tfrac12,1)$.  We integrate the inequality over $(0,t)$ and exploit the representation $\pta v = \pt (g_{1-\alpha} * v)$ of the Riemann--Liouville derivative, which gives
\begin{equation} \label{Eq:EnergyFP}\begin{aligned} & \frac12 \big(\gb* \|\hphi_k-\init \ga\|_\hY^2\big)(t)  +\frac{1}{4\lambda} \|\nablaq \hphi_k\|_{L^2_t\hY}^2 + \frac{\eps}{2} \| \nablax \hphi_k\|_{L^2_t \hY}^2 - \delta \|\hphi_k\|^2_{L^2_t\hY}  \\
&\quad \leq C(\delta,\alpha) \|M^{1/2} \init\|_{H^1(D;W^{1,\infty}(\Omega)) }^2  \int_0^t g_{2\alpha-1}(s) \|u_k(s)\|_{\mathcal{H}}^2 \ds
+ C(\delta)\| \init\|_\hX^2  g_{2\alpha}(T).
\end{aligned}\end{equation}
Further, we derive a lower bound on the first term of the left-hand side by noting that $(g_1*v)(t) \leq T^{\alpha} \Gamma(1-\alpha) (\gb*v)(t)$, see \cref{Eq:KernelNorm}, and therefore we have that
$$\begin{aligned} &\frac12 \big(\gb* \|\hphi_k-\init \ga\|_\hY^2\big)(t) \\ &\quad \geq \frac{T^{-\alpha}}{2\Gamma(1-\alpha)} \int_0^t  \|\hphi_k(s)-\init \ga(s)\|_{\hY}^2 \ds \\ &\quad \geq \frac{T^{-\alpha}}{2\Gamma(1-\alpha)} \int_0^t \big|\,\|\hphi_k(s)\|_\hY-\ga(s)\|\init \|_{\hY} \big|^2\ds \\
&\quad =\frac{T^{-\alpha}}{2\Gamma(1-\alpha)} \int_0^t \|\hphi_k(s)\|_\hY^2 - 2\ga(s)\|\hphi_k(s)\|_\hY\|\init \|_{\hY} +g_{\alpha}^2(s)\|\init \|_{\hY}^2\ds,
\end{aligned} $$
where we applied the reverse triangle inequality in the second estimate.
The function $g_{\alpha}$ belongs to $L^2(0,t)$ for any $\alpha\in (\tfrac12,1)$ and the integral of $g_\alpha^2$ is positive. We apply H\"older's inequality in the second term of the integrand and note that the $L^2(0,t)$-norm of $g_\alpha$ has the upper bound $C(\alpha)T^{\alpha-1/2}$. We thus have that
$$\begin{aligned}   \frac12 \big(\gb* \|\hphi_k-\init \ga\|_\hY^2\big)(t)   &\geq 
\frac{T^{-\alpha}}{2\Gamma(1-\alpha)} \left[
\frac12 \|\hphi_k\|_{L^2_t\hY}^2 - 2 C(\alpha) T^{\alpha-1/2} \|\init \|_{\hY}  \|\hphi_k\|_{L^2_t\hY} \right]\\ &\geq \frac{T^{-\alpha}}{2\Gamma(1-\alpha)}\left[\frac14 \|\hphi_k\|_{L^2_t\hY}^2 - C(T,\alpha) \|\init \|_{\hY}^2\right],
\end{aligned} $$
where we have applied Young's inequality in the last step. We multiply the energy estimate \cref{Eq:EnergyFP} by $\gamma>0$ and obtain for $\delta=\frac{T^{-\alpha}}{16 \Gamma(1-\alpha)}$ the estimate \cref{Est:SolFP}.
\end{proof} %\medskip

Next, we combine the estimates on $u_k$ and $\hphi_k$, see \cref{Lem:EstU} and \cref{Lem:EstPhi}, to obtain a $k$-uniform bound. %\medskip

\begin{lemma} \label{Lem:EstComb}
	Let \cref{Ass:WellPosedness} hold; then, the following $k$-uniform estimate on the Galerkin solution $(u_k,\hphi_k)$ holds:
	\begin{equation} \label{Eq:EnergyIneq}
		\begin{aligned} &  \big(\gb* \|\hphi_k-\init \ga\|_\hY^2\big)(t) +  \|\hphi_k\|_{L^2_t\hX}^2    + \|u_k(t)\|_{\mathcal{H}}^2 +  \|\nabla u_k\|_{L^2_t{\mathcal{H}}}^2 \\
			&\quad \leq C\big(\alpha,T, \|u^0\|_{\mathcal{H}}^2,\|M^{1/2} \hpsi^0\|^2_{H^1(D;H^{1+d/2+\delta}(\Omega)) }\big).
	\end{aligned}\end{equation}
\end{lemma} 

\begin{proof}
We add the integrated velocity inequality \cref{Eq:EnergyU2} to the bound \cref{Est:SolFP} and obtain the following combined bound:
$$\begin{aligned} &\frac{\gamma}{2} \big(\gb* \|\hphi_k-\init \ga\|_\hY^2\big)(t) + \frac{\gamma\, T^{-\alpha}}{16\,\Gamma(1-\alpha)} \|\hphi_k\|_{L^2_t\hY}^2  +\frac{\gamma}{2\lambda} \|\nablaq \hphi_k\|_{L^2_t\hY}^2 + \gamma \eps \| \nablax \hphi_k\|_{L^2_t \hY}^2  \\ &\quad + \frac12 \|u_k(t)\|_{\mathcal{H}}^2 + \frac{\nu}{2} \|\nabla u_k\|_{L^2_t{\mathcal{H}}}^2 \\
&\leq C(\alpha,\gamma) \|M^{1/2} \init\|_{H^1(D;W^{1,\infty}(\Omega)) }^2  \int_0^t g_{2\alpha-1}(s) \|u_k(s)\|_{\mathcal{H}}^2 \ds
+ C(\alpha,\gamma,T)\| \init\|_\hX^2 \\&\quad + \frac12 \|u_k^0\|_{\mathcal{H}}^2 + %
\frac{ T^{2-2\alpha} k_B^2\mu_T^2}{2\nu (1-\alpha)^2}  \|\hphi_k\|_{L^2_t\hY}^2.
\end{aligned}$$
We now choose $\gamma$ such that 
$$\frac{\gamma\, T^{-\alpha}}{16\, \Gamma(1-\alpha)} \geq \frac{ T^{2-2\alpha} k_B^2\mu_T^2}{\nu (1-\alpha)^2}.$$
Hence, we can absorb the last term on the right-hand side into the second term on the left-hand side, and 
 the combined energy inequality thus becomes
\begin{equation} \label{Eq:EnergyBack}\begin{aligned} & \big(\gb* \|\hphi_k-\init \ga\|_\hY^2\big)(t) +  \|\hphi_k\|_{L^2_t\hX}^2    + \|u_k(t)\|_{\mathcal{H}}^2 +  \|\nabla u_k\|_{L^2_t{\mathcal{H}}}^2 \\
&\quad\leq C(\alpha,T) \|M^{1/2} \init\|_{H^1(D;W^{1,\infty}(\Omega)) }^2  \int_0^t g_{2\alpha-1}(s) \|u_k(s)\|_{\mathcal{H}}^2 \ds
\\ &\qquad + C(\alpha,T) \big( \| \init\|_\hX^2 + \|u_k^0\|_{\mathcal{H}}^2\big),
\end{aligned}\end{equation}
where we took the minimum of each prefactor of the norms on the left-hand side and divided the inequality by this value. Gronwall's lemma then implies that
\begin{equation} 
\begin{aligned} \label{Eq:EnergyBack1}
& \big(\gb* \|\hphi_k-\init \ga\|_\hY^2\big)(t) +  \|\hphi_k\|_{L^2_t\hX}^2    + \|u_k(t)\|_{\mathcal{H}}^2 +  \|\nabla u_k\|_{L^2_t{\mathcal{H}}}^2 \\
&\quad \leq C(\alpha,T) \cdot \big(  \| \init\|_\hX^2  + \|u_k^0\|_{\mathcal{H}}^2  \big) \cdot \textup{exp}\bigg(\frac{T^{2\alpha-1}}{2\alpha-1} \|M^{1/2} \init\|^2_{H^1(D;W^{1,\infty}(\Omega)) }\bigg).
\end{aligned}
\end{equation}

We note that the initial conditions of the Galerkin system are defined by $u_k^0=\Pi_{\mathcal{H}_k} u^0$ and $\init=\Pi_{\hY_k} \hpsi^0$. Therefore, we have that $\|u_k^0\|_{\mathcal{H}}^2 \leq \|u^0\|_{\mathcal{H}}^2$ and $$\|M^{1/2}\init\|_{H^1(D;W^{1,\infty}(\Omega)) }^2 \leq C \|M^{1/2}\hpsi^0\|_{H^1(D;H^{1+d/2+\delta}(\Omega)) }^2.$$
%see \cite[Theorem 8.1.11]{brenner2008mathematical} with regards to the stability of the Ritz projection in $W^{1,\infty}$.
We insert these bounds into the right-hand side of the inequality \cref{Eq:EnergyBack1} and we thus arrive at the desired $k$-uniform energy estimate \cref{Eq:EnergyIneq}.
\end{proof}

\subsection{Convergence of subsequences} 
Having derived the $k$-uniform energy estimate \cref{Eq:EnergyIneq} stated in \cref{Lem:EstComb},
we shall extract weakly/weakly-$*$ converging subsequences of Galerkin solutions $(u_k,\hphi_k)$. We shall also prove strong convergence of a subsequence $u_{k_j}$ in $L^2(0,T;\mathcal{H}_0)$ in order to pass to the limit $j\to \infty$ in the nonlinear terms in the variational Navier--Stokes--Fokker--Planck system.   %\medskip

\begin{lemma} Let \cref{Ass:WellPosedness} hold and assume that $r \in [1,\infty)$ for $d=2$ and $r \in [1,6)$ for $d=3$; then, the sequence of Galerkin solutions $(u_k,\hphi_k)$ from \cref{Eq:LemmaExistenceODE} contains a subsequence $(u_{k_j},\hphi_{k_j})$ that admits the following convergences as $j \to \infty$:
	%\footnote{\color{red}~I don't see where the 7th weak convergence, with spatial function space $\hY$, is coming from. Surely this is incorrect. -- This should be corrected now} 
	\begin{equation} \label{Eq:Weak} \begin{aligned}	
			u_{k_j} &\longweak u &&\text{weakly-$*$ in } L^\infty(0,T;\mathcal{H}_0), \\
			u_{k_j} &\longweak u &&\text{weakly\phantom{-*} in } L^2(0,T;\mathcal{V}_0) \cap L^{8/d}(0,T;L^4(\Omega)^d), \\
			\hphi_{k_j} &\longweak 		 \hphi &&\text{weakly\phantom{-*} in }
			L^2(0,T;\hX), \\
		\pt u_{k_j} &\longweak \pt u &&\text{weakly\phantom{-*} in } L^{8/(4+d)}(0,T;\mathcal{V}_0'), \\
		u_{k_j} &\longrightarrow  u &&\text{strongly\hspace{1mm} in } L^2\big(0,T;L^{r}(\Omega;\R^d)\big), \\
		u_{k_j} &\longrightarrow  u &&\text{strongly\hspace{1mm} in } C([0,T];\mathcal{V}_0'),\\
		\pt^{\alpha} \hphi_{k_j} &\longweak \pt^{\alpha} \hphi &&\text{weakly\phantom{-*} in } L^{8/(4+d)}(0,T;\hZ'),\\
		\hphi_{k_j} &\longrightarrow  \hphi &&\text{strongly\hspace{1mm} in } L^2(0,T;\hY),
\\ 
\C(M \gb * \hphi_{k_j}) &\longrightarrow \C(M\gb *\hphi) &&\text{strongly\hspace{1mm} in }
		L^2\big(0,T;L^2(\Omega;\R^{d\times d})\big), \\
		\gb*\hphi_{k_j} &\longweak \gb *\hphi &&\text{weakly-$*$ in } L^\infty(0,T;\hY) \cap L^2(0,T;\hX), \\
		\gb*\hphi_{k_j} &\longrightarrow \gb *\hphi &&\text{strongly\hspace{1mm} in }
		C([0,T];\hX') \cap L^2(0,T;\hY).
		\end{aligned}
	\end{equation} 
	\end{lemma}
\begin{proof}
In Lemma \ref{Lem:EstComb} we stated various $k$-uniform bounds on $u_k$ and $\hphi_k$. Thanks to the Banach--Alaoglu and Eberlein--\v{S}mulian theorems, see \cite[Theorem 8.10]{alt2016linear}, there are weakly/weakly-$*$ converging subsequences $u_{k_j}$ and $\hphi_{k_j}$. In particular, we obtain the convergences
\begin{equation} \label{Eq:Weak1} \begin{aligned}	
u_{k_j} &\longweak u &&\text{weakly-$*$ in } L^\infty(0,T;\mathcal{H}_0), \\
u_{k_j} &\longweak u &&\text{weakly\phantom{-*} in } L^2(0,T;\mathcal{V}_0), \\
\hphi_{k_j} &\longweak 		 \hphi &&\text{weakly\phantom{-*} in }
L^2(0,T;\hX). 
\end{aligned}\end{equation}

We shall establish the strong convergence of $u_{k_j}$ in $L^2(0,T;\mathcal{H}_0)$ by applying the Aubin--Lions compactness lemma; see \cref{Eq:aubin}. To this end, we need to bound the time derivative of $u_k$ in a suitable dual space. Let us therefore consider an arbitrary element $v \in L^{8/(4-d)}(0,T;\mathcal{V}_0)$ and bound each of the terms appearing on the right-hand side of \cref{Eq:NS_dis} below by means of H\"older's inequality:
$$\begin{aligned}\int_0^T | \langle \pt u_k,v \rangle_{\mathcal{V}_0}| \dt   &= \int_0^T  \Big| -((u_k \cdot \nablax)u_k,\Pi_{\mathcal{H}_k} v)_{\mathcal{H}}  \\ 
&\quad  - \nu (\nablax u_k,\nablax \Pi_{\mathcal{H}_k} v)_{\mathcal{H}} - k_B \mu_T \big(\C(M\gb*\hphi_k),\nablax \Pi_{\mathcal{H}_k} v \big)_{\mathcal{H}} \Big| \dt
	\\ 
& \leq C \int_0^T \Big( \|u_k\|_{L^4(\Omega)} \|u_k\|_{\mathcal{V}} \|\Pi_{\mathcal{H}_k} v\|_{L^4(\Omega)}  \\ &\quad  + \|u_k\|_{\mathcal{V}} \|\Pi_{\mathcal{H}_k} v\|_{\mathcal{V}} +  \|\C(M\gb*\hphi_k)\|_{\mathcal{H}}  \|\Pi_{\mathcal{H}_k} v\|_{\mathcal{V}} \Big) \dt. 
\end{aligned}$$
Hence, using Ladyzhenskaya's inequality, we have that
$$\begin{aligned}
&\int_0^T | \langle \pt u_k,v \rangle_{\mathcal{V}_0}| \dt\\  
 &\leq C \int_0^T \Big( \|u_k\|_{\mathcal H}^{1-d/4} \|u_k\|_{\mathcal{V}}^{1+d/4} \|\Pi_{\mathcal{H}_k} v\|_{L^4(\Omega)} + \|u_k\|_{\mathcal{V}} \|\Pi_{\mathcal{H}_k} v\|_{\mathcal{V}} +  \|\hphi_k\|_\hY  \|\Pi_{\mathcal{H}_k} v\|_{\mathcal{V}} \Big) \dt 
 \\
&\leq C \Big( \|u_k\|_{L^\infty {\mathcal{H}}}^{1-d/4} \|u_k\|_{L^2{\mathcal{V}}}^{1+d/4} \|v\|_{L^{8/(4-d)} {\mathcal{V}}} + \|u_k\|_{L^2{\mathcal{V}}} \|v\|_{L^2{\mathcal{V}}} +  \|\hphi_k\|_{L^2\hY} \|v\|_{L^2{\mathcal{V}}} \Big)
	\\
	&\leq C\|v\|_{L^{8/(4-d)} {\mathcal{V}}}.
\end{aligned}$$
This then implies that $\pt u_k$ is bounded in $L^{8/(4+d)}(0,T;\mathcal{V}_0
 ')$. 
It follows by the Aubin--Lions lemma \cref{Eq:aubin} that  %
\begin{equation} \label{Eq:Weak2} \begin{aligned}	
\pt u_{k_j} &\longweak \pt u &&\text{weakly\phantom{-*} in } L^{8/(4+d)}(0,T;\mathcal{V}_0'), \\
u_{k_j} &\longrightarrow  u &&\text{strongly\hspace{1mm} in } L^2\big(0,T;L^{r}(\Omega;\R^d)\big), \\
u_{k_j} &\longrightarrow  u &&\text{strongly\hspace{1mm} in } C([0,T];\mathcal{V}_0'),
\end{aligned}\end{equation}
where $r \in [1,\infty)$ for $d=2$ and $r \in [1,6)$ for $d=3$.

Similarly, we consider an arbitrary element $\hzeta \in L^\frac{8}{4-d}(0,T;\hZ)$ and we recall that $\hZ$ was defined in the beginning of \cref{Sec:Maxwell}. We test the Galerkin equation of $\hphi_k$ with $\Pi_{\mathcal H_k} \hzeta$ giving
\begin{equation} \label{Eq:BoundDerivative} \begin{aligned}\int_0^T \!\! |\langle \pta \hphi_k,\Pi_{\mathcal H_k} \hzeta \rangle_{\hX}| \dt  &=\!
\int_0^T\!\! \Big| -((u_k \cdot \nablax) \hphi_k, \Pi_{\mathcal H_k} \hzeta)_\hY - \frac{1}{2\lambda} (\nablaq \hphi_k,\nablaq \Pi_{\mathcal H_k} \hzeta)_\hY   \\ &\quad - \eps (\nablax \hphi_k,\nabla_x \Pi_{\mathcal H_k} \hzeta)_\hY+ (\omega(u_k) q \hphi_k, \nablaq \Pi_{\mathcal H_k} \hzeta)_\hY \Big|\d t.\end{aligned} \end{equation}
We note that $u_k$ is bounded in $L^{8/d}(0,T;L^4(\Omega)^d)$ by the following interpolation result
$$\int_0^T \|u_k\|_{L^4}^{8/d} \dt \leq \int_0^T  \|u_k\|_{\mathcal H}^{8/d-2} \|u_k\|^{2}_{\mathcal V} \dt \leq \|u_k\|_{L^\infty \mathcal{H}}^{8/d-2} \|u_k\|_{L^{2} \mathcal{V}}^{2}.$$
Regarding the last term in \eqref{Eq:BoundDerivative}, we integrate by parts and estimate by H\"older's inequality
$$\begin{aligned} -  &\big(\omega(u_k)  q \hphi, \nablaq \hzeta\big)_\hY \\ &=  \big(\nablax \hphi_k (\nablaq \hzeta)^{\mathrm{T}} q,u_k\big)_\hY + \big( \hphi_k \nablax \nablaq \hzeta q,u_k\big)_\hY + \big(u_k \cdot q, \nablax \hphi_k \cdot \nablaq \hzeta\big)_\hY \\ &\quad  + \big(u_k \cdot q, \hphi_k \div_x \nablaq \hzeta\big)_\hY \\
	&\leq C  \| \nablax \hphi_k\|_{L^2\hY}
	\big( \|\nablaq \hzeta\|_{L^{8/(4-d)} \hY}  +  \|\nablax \nablaq \hzeta\|_{L^{8/(4-d)} \hY}  \big) \|q\|_{L^\infty} \|u_k\|_{L^{8/d} L^4}  \\ &\quad+ \| \hphi_k\|_{L^2\hX} \|\nablax \nablaq \hzeta\|_{{L^{8/(4-d)} \hY}} \|q\|_{L^\infty} \|u_k\|_{L^{8/d} L^4}  \\ &\quad + C \|u_k\|_{L^{8/d} L^4} \|q\|_{L^\infty} \|\nablax \hphi_k\|_{L^2 \hY} \big(\|\nablaq \hzeta\|_{L^{8/(4-d)}\hY}  + \|\nablax \nablaq \hzeta\|_{L^{8/(4-d)}\hY} \big) \\ &\quad +\|u_k\|_{L^{8/d} L^4} \|q\|_{L^\infty}  \|\hphi_k\|_{L^2 \hX} \|\div_x \nablaq \hzeta\|_{L^{8/(4-d)}\hY}  \\
	&\leq C \|u_k\|_{L^{8/d} L^4} \|q\|_{L^\infty}  \|\hphi_k\|_{L^2 \hX} \| \hzeta\|_{L^{8/(4-d)}(0,T;\hZ) } .
\end{aligned} $$
Using this estimate, we can bound \eqref{Eq:BoundDerivative} as follows:
\begin{equation} \label{Eq:BoundDerivativePhi} \begin{aligned}& \int_0^T |\langle \pta \hphi_k,\Pi_{\mathcal H_k} \hzeta \rangle_{\hX}| \dt  \\ &\leq C \Big( \|u_k \|_{L^\infty\mathcal{H}_0 }  \|\hphi_k \|_{L^2 \hX}   \|\hzeta \|_{L^2\hX} +  \|\nablaq \hphi_k\|_{L^2\hY}   \|\nablaq \hzeta \|_{L^2\hY}    \\ &\quad +  \| \nablax \hphi_k\|_{L^2\hY}   \|\nabla_x\hzeta \|_{L^2\hY} + \|u_k\|_{L^{8/d} L^4} \|q\|_{L^\infty}  \|\hphi_k\|_{L^2 \hX} \| \hzeta\|_{L^{8/(4-d)}(0,T;\hZ) } \Big) \\
		&\leq  C \|\hzeta\|_{L^{8/(4-d)}(0,T;\hZ) } .
\end{aligned} \end{equation}
Hence, we obtain the $k$-uniform boundedness of $\pt(\gb*\hphi_{k})=\pt^{\alpha} \hphi_k$ in the space $L^{8/(4+d)}(0,T;\hZ')$, which is continuously embedded  in $L^{8/(4+d)}(0,T;(H^2_M(\Omega \times D))')$. 
Therefore, we are in the setting of the Gelfand triple
$$\hX \com \hY\hookrightarrow \big( H_M^2(\Omega \times D)\big)'.$$ 
%and $\hphi_{k}$ is bounded in the space 
%$$L^2(0,T;\hX) \cap W^{\alpha,8/(4+d)}(0,T;(H^2(\Omega \times D;M))').$$
We thus obtain from the fractional Aubin--Lions lemma, see \eqref{Eq:aubinfractional2}, that
\begin{equation} \label{Eq:Weak3} \begin{aligned}	
		\pt^{\alpha} \hphi_{k_j} &\longweak \pt^{\alpha} \hphi &&\text{weakly\phantom{-*} in } L^{8/(4+d)}(0,T;\hZ'), \\
		\hphi_{k_j} &\longrightarrow  \hphi &&\text{strongly\hspace{1mm} in } L^2(0,T;\hY).
\end{aligned} \end{equation}

%We note that we can test the weak form of $\hphi_k$ again by $\hzeta=g_\alpha * \pta \hphi_k$, i.e., we consider the tested weak form \cref{Eq:TestingRHS}. However, this time we do not apply the fractional chain inequality on the first term on the left-hand side of \cref{Eq:TestingRHS} but we exploit the coercivity of the kernel function in the integrated weak form, see
% \cref{Eq:Coercive2}, to obtain
% $$\int_0^t (\pta \hphi_k,\hphi_k - \ga \hpsi_k^0)_\hY \ds  \geq  \cos(\alpha \pi/2)  \Big( \tfrac12 \|\pt^{\alpha/2} \hphi_k\|^2_{L^2_t \hY}-\tfrac{\Gamma(\alpha-1)}{\Gamma(\alpha/2)^2}g_{\alpha}(t) \|\hpsi_k^0\|_\hY^2 \Big). $$
  %We take $t=T$ and use the derived energy estimates to bound the right-hand side of \cref{Eq:TestingRHS} similiarly to before. 
 % We can apply the fractional Aubin--Lions lemma, see \eqref{Eq:aubinfractional2}, to obtain the strong convergence of $\hphi_{k_j}$ in $L^2(0,T;\hY)$; in other words, 
The convolution $\gb*\hphi_{k}$ is bounded in $L^2(0,T;\hX)$ thanks to Young's convolution inequality 
$$\|\gb * \hphi_k\|_{L^2_t\hX} \leq \|g_{1-\alpha}\|_{L^1_t} \|\hphi_{k}\|_{L^2_t\hX} \leq C T^{1-\alpha} \|\hphi_{k}\|_{L^2_t\hX}.$$
Moreover,  $\gb*\hphi_{k}$ is bounded in $L^\infty(0,T;\hY)$ by the following chain of estimates:
$$\begin{aligned} &\|\gb*\hphi_{k}\|_{L^\infty \hY} \\&\quad \leq \sup_{t \in (0,T)} \int_0^t \gb(t-s) \|\hphi_{k}(s)\|_\hY \ds \\ 
	&\quad \leq \sup_{t \in (0,T)} \int_0^t \gb(t-s) \|\hphi_{k}(s)-\hpsi_{k}^0 \ga(s)\|_\hY \ds + (\gb*\ga)(t) \|\hpsi_{k}^0\|_\hY   \\ 
	&\quad \leq   \sup_{t \in (0,T)} \int_0^t \gb(t-s) \|\hphi_{k}(s)-\hpsi_{k}^0 \ga(s)\|^2_\hY \ds + \frac14 \int_0^t \gb(t-s) \ds + \|\hpsi_{k}^0\|_\hY  
	\\ &\quad = \sup_{t \in (0,T)} (\gb*\|\hphi_{k}-\hpsi_{k}^0 \ga\|_\hY^2)(t) + \frac14 g_{2-\alpha}(T) + \|\hpsi_{k}^0\|_\hY,
\end{aligned}$$
and the first term on the right-hand side is bounded by \cref{Lem:EstComb}. Since we have already proved a bound on $\pt(\gb*\hphi_{k})=\pta \hphi_{k}$, see \eqref{Eq:BoundDerivativePhi}, we may use the Aubin--Lions lemma, see \cref{Eq:aubin}, to obtain the following strong convergence results: 
\begin{equation} \label{Eq:Weak5} \begin{aligned}	
		\gb*\hphi_{k_j} &\longrightarrow  \gb*\hphi &&\text{strongly\hspace{1mm} in } L^2(0,T;\hY), \\
		\gb*\hphi_{k_j} &\longrightarrow  \gb*\hphi &&\text{strongly\hspace{1mm} in } C([0,T];\hX').
\end{aligned} \end{equation}
Lastly, we note that the mapping $M\gb*\varphi \mapsto \C(M\gb*\varphi)$ is linear and continuous thanks to \eqref{Eq:C},  %\leq C \|g_{1-\alpha}\|_{L^1_t} \|\varphi\|_{L^2_t\hY} \leq C T^{1-\alpha} \|\varphi\|_{L^2_t\hY}, $$
and therefore we have from \eqref{Eq:Weak5}$_1$ that
\begin{equation} \label{Eq:Weak4} \C(M \gb * \hphi_{k_j}) \longrightarrow \C(M\gb *\hphi) \quad \text{ strongly in }
L^2\big(0,T;L^2(\Omega;\R^{d\times d})\big).
\end{equation} 
\end{proof}

\subsection{Passage to the limit}
Next, we pass to the limit $j \to \infty$ in the  time-integrated $k_j$-th Galerkin system \cref{Eq:NS_dis}, \cref{Eq:FP_dis}. Specifically, we shall use the convergence results stated in the preceding lemma to show that the weak limits, $u$ and $\hphi$, satisfy the variational Navier--Stokes--Fokker--Planck system in the sense of \cref{Eq:DefWeak}. 
\begin{proof}[Proof of \cref{Thm:WellPosedness}]
We consider the time-integrated Galerkin system
\begin{align} 
&\int_0^T \Big(\langle\pt u_{k_j},v \rangle_V  + ((u_{k_j} \cdot \nablax)u_{k_j},v)_{\mathcal{H}} \label{Eq:NS_dis_time} \\ 
& \quad + \nu (\nablax u_{k_j},\nablax v)_{\mathcal{H}}+ k \mu (C(M \gb *\hphi_{k_j}),\nablax v)_{\mathcal{H}} \Big) \eta(t) \dt =0  \notag\\
&\int_0^T  -(\gb*\hphi_{k_j},\hzeta)_\hY \eta'(t) + \Big(    ((u_{k_j}\cdot \nablax) \hphi_{k_j}, \hzeta)_\hY+ \frac{1}{2\lambda} ( \nablaq \hphi_{k_j},\nablaq \hat\zeta)_\hY \label{Eq:FP_dis_time} \\
&\quad + \eps (\nablax \hphi_{k_j},\nabla \hat\zeta)_\hY- (\omega(u_{k_j}) q \hphi_{k_j}, \nablaq \hat\zeta)_\hY \Big) \eta(t) \dt=0, \notag \end{align} for all $v \in \H_{k_j}$, $\eta \in C_0^\infty(0,T)$ and $\hat\zeta \in \hY_{k_j}$.  
Passing to the limit $j \rightarrow \infty$ in \eqref{Eq:NS_dis_time} using \cref{Eq:Weak1}--\cref{Eq:Weak4}
is standard, and results in \eqref{Eq:NS}. It therefore remains to pass to the limit $j \rightarrow \infty$ in \eqref{Eq:FP_dis_time}. In particular, the convergence of the linear terms follow immediately by weak convergence, and we only consider the two nonlinear terms. 
 We note that $\hphi_{k_j} \to \hphi$ strongly in $L^2\big(0,T;\hY)$ and $\omega(u_{k_j}) \to \omega(u)$ weakly in $L^2(0,T;\H_0)$, from which we deduce that
$$\int_0^T (\omega(u_{k_j}) q \hphi_{k_j}, \nablaq \hat\zeta)_\hY  \eta(t) \dt \longrightarrow  \int_0^T (\omega(u) q \hphi, \nablaq \hat\zeta)_\hY  \eta(t) \dt,$$
as $j \to \infty$. With the same reasoning, we are able to show that
$$\int_0^T  ((u_{k_j}\cdot \nablax) \hphi_{k_j}, \hat\zeta)_{\hY} \eta(t) \dt \longrightarrow \int_0^T  ( (u \cdot\nablax) \hphi, \hat\zeta)_{\hY} \eta(t) \dt\quad \mbox{as $j \to \infty$}.$$

We use the density of $\cup_{k=1}^\infty \H_{k}$ in $\mathcal{V}$ and of $\cup_{k=1}^\infty \hY_{k}$ in $H^2_M(\Omega \times D)$, which completes the proof by observing that the tuple $(u,\hphi)$ satisfies the variational form of the time-fractional Navier--Stokes--Fokker--Planck system as stated in \cref{Eq:DefWeak}. 

It remains to check that the initial conditions are satisfied. First, we obtain the convergence $u_{k_j}(0) \to u(0)$ in $\mathcal{V}_0'$  as $j \rightarrow \infty$; see again \cref{Eq:Weak}. However, by definition, $u_{k_j}(0)=\Pi_{\H_{k_j}} u^0$, which converges to $u^0$ in $\mathcal{H}_0$
as $j \rightarrow \infty$. By the uniqueness of the limit it follows that $u(0)=u^0$. Regarding the solution of the Fokker--Planck equation, we use again the strong convergence \cref{Eq:Weak} to conclude $(\gb*\hphi)(0) = \hpsi^0$.
\end{proof}

	Having proved that a weak solution tuple $(u,\hphi)$ to the time-fractional system in the sense of \cref{Eq:DefWeak} exists, we return to the original variable $\psi:=\gb* ( M \hphi)$, whose evolution is governed by the time-fractional Fokker--Planck equation \cref{Eq:DerivFP}. Indeed, $\gb*\psi=g_{2-2\alpha}*\phi$, which is continuous for $\alpha> 1/2$ and $(\gb*\psi)(0)=0$ as we have originally assumed in the model transformation. In this sense, we have also shown the existence of a variational solution tuple $(u,\psi)$ to the original time-fractional model.

\section*{Conclusions and outlook}
	In this paper, we investigated the well-posedness of a coupled Navier--Stokes--Fokker--Planck system with a time-fractional derivative. Such systems arise in the kinetic theory of polymeric liquid solutions with noninteracting polymer chains. We outlined the derivation of the model from a subordinated Langevin equation and considered the case of a finitely extensible nonlinear elastic (FENE-type) dumbbell model with a corotational drag term. We proved the existence of large-data global-in-time weak solutions to the corotational time-fractional model of order $\alpha \in (\tfrac12,1)$ and derived a uniform energy inequality by considering a nonstandard and novel testing procedure.
The existence of weak solutions to the general noncorotational time-fractional FENE model is an open problem, which will be studied in a forthcoming paper by using a different testing procedure; see \cite{barrett2011existence} for the integer-order setting (corresponding to $\alpha=1$). Concerning the numerical approximation of the time-fractional system considered here we refer the reader to the recent paper
%We bear in mind that numerical simulations have been omitted to emphasize the analytical findings.
%To discretize a system in both $\Omega$ and $D$, one must explore strategies such as the heterogeneous alternating-direction method, as done in \cite{knezevic2009heterogeneous} for a FENE-type dumbbell model.
%Simultaneously, the fractional derivative has to be discretized, and we recommend a modern approach to avoid the costly memory effect, e.g., the rational approximation scheme \cite{khristenko}. This is done by the submitted paper 
\cite{beddrich2023numerical}.

{\small	
	\bibliography{literature.bib}
	\bibliographystyle{siamplain} }
	
\end{document}